\documentclass{amsart}
\usepackage{amsmath}
\usepackage{amssymb}
\usepackage{amsthm}
\usepackage{dsfont}

\usepackage{graphicx}
\usepackage{subfigure}
\graphicspath{{figures/}}

\usepackage{siunitx}

\usepackage{enumerate}

\usepackage{enumitem}

\usepackage{bbold}

\usepackage{bm}

\usepackage[colorlinks = true, linkcolor = blue, citecolor = blue, urlcolor = blue]{hyperref}

\usepackage{mathtools}

\usepackage{algorithm}
\usepackage{algorithmic}
\usepackage{ragged2e}

\usepackage{tikz}
\usetikzlibrary{fit}

\usepackage{multirow}

\newtheorem{theorem}{Theorem}[section]

\newtheorem{proposition}[theorem]{Proposition}

\theoremstyle{remark}
\newtheorem{remark}{Remark}[section]
\theoremstyle{definition}
\newtheorem{example}{Example}[section]

\newcommand{\Z}{\mathbb{Z}}
\newcommand{\R}{\mathbb{R}}

\newcommand{\Rd}{\R^d}

\newcommand{\bO}{\mathcal{O}}

\newcommand{\bq}{\begin{equation}}
\newcommand{\eq}{\end{equation}}

\newcommand{\abs}[1]{\vert#1\vert}
\newcommand{\Abs}[1]{\left\vert#1\right\vert}

\usepackage{xspace}

\DeclareMathOperator{\argmin}{argmin}

\usepackage{caption}

\newcommand{\boxalgorithm}[2]{
\begin{algorithm}#1
\begin{justifying}\noindent
#2
\end{justifying}
\end{algorithm}}

\newcommand{\dx}{\:dx}

\begin{document}

\title{A simplified threshold dynamics algorithm for isotropic surface energies}

\author{Tiago Salvador}
\address{Department of Mathematics, University of Michigan, 530 Church St. Ann Arbor, MI 48105 ({\tt saldanha@umich.edu}).}

\author{Selim Esedo\={g}lu}
\address{Department of Mathematics, University of Michigan, 530 Church St. Ann Arbor, MI 48105 ({\tt esedoglu@umich.edu}).}

\date{\today}

\begin{abstract}
We present a simplified version of the threshold dynamics algorithm given in \cite{SelimFelix}.
The new version still allows specifying $N\choose 2$ possibly distinct surface tensions and $N\choose 2$ possibly distinct mobilities for a network with $N$ phases, but achieves this level of generality without the use of retardation functions.
Instead, it employs linear combinations of Gaussians in the convolution step of the algorithm.
Convolutions with only two distinct Gaussians is enough for the entire network, maintaining the efficiency of the original thresholding scheme.
We discuss stability and convergence of the new algorithm, including some counterexamples in which convergence fails.
The apparently convergent cases include unequal surface tensions given by the Read \& Shockley model and its three dimensional extensions, along with equal mobilities, that are a very common choice in computational materials science.
\end{abstract}



\maketitle

\setcounter{tocdepth}{1}

\section{Introduction}

Threshold dynamics - also known as diffusion or convolution generated motion - is a very efficient algorithm originally proposed by Merriman, Bence and Osher (MBO) in \cite{MBO92,MBO94} for simulating the motion by mean curvature flow of an interface. It is based on the observation that the level-set of a distance function or characteristic function evolved under the heat equation moves in the normal direction with velocity equal to the mean curvature of the level-set surface. The method alternately diffuses (through convolution with a kernel) and sharpens characteristic functions of regions (by pointwise thresholding). In its simplest form (for isotropic, two-phase mean curvature flow), it is given as follows:
\boxalgorithm{\caption{(in \cite{MBO92})}\label{alg:MBO}}{Given the initial condition $\Sigma^0$ and time step size $\delta t$, to obtain the approximate solution $\Sigma^{k+1}$ at time $t = (\delta t)(k+1)$ from $\Sigma^k$ at time $t = (\delta t) k$, alternate the following steps:
\begin{enumerate}
	\item Convolution step:
\[
\psi^k = \frac{1}{(\delta t)^{\frac{d}{2}}}K\left(\frac{\cdot}{\sqrt{\delta t}}\right) \ast \textbf{1}_{\Sigma^k}.
\]
	\item Thresholding set:
\[
\Sigma^{k+1} = \left\{x: \psi^k(x) \geq \frac{1}{2}\int_{\R^d} K(x)\dx\right\}.
\]
\end{enumerate}}

\noindent
The convolution kernel $K:\Rd\to\R$ was chosen in \cite{MBO94} to be the Gaussian
\bq\label{eq:Gaussian}
G(x) = \frac{1}{(4\pi)^{\frac{d}{2}}}\exp\left(-\frac{\abs{x}^2}{4}\right),
\eq
but the possibility of choosing other kernels is also mentioned in \cite{MBO92}. For this particular choice of kernel, the boundary of the set $\partial \Sigma^k$ can be shown to evolve, to leading order, by mean curvature motion; see e.g. \cite{RuuthEfficientAlgorithms} for a truncation error analysis, and e.g \cite{EvansMC,IshiiTDPropagatingFronts} for proofs of convergence. In particular in \cite{IshiiTDPropagatingFronts},  positivity of the kernel is essential since it guarantees that both steps of Algorithm \ref{alg:MBO} are monotone, thereby allowing the scheme to enjoy a comparison principle which is a key tool in the convergence proof.

Among the benefits of the MBO algorithm are (i) implicit representation of the interface as in the phase field or level set methods, allowing for graceful handling of topological changes, (ii) unconditional stability, where the time step size is restricted only by accuracy considerations, and (iii) very low per time step cost when implemented on uniform grids.

Motion by mean curvature arises as $L^2$ gradient descent for perimeter of sets. Perimeter of sets, in turn, are key to variational models for interfaces in a great variety of applications, ranging from image processing and computer vision (e.g. the Mumford-Shah model \cite{MumfordShahImageSegmentation} for image segmentation) to materials science (e.g. Mullins' model \cite{MullinsGrainBoundaryMotion} for grain boundary motion in polycrystals). More recently, such variational models and their minimization via gradient descent have also been applied in the context of machine learning and artificial intelligence (e.g. graph partitioning models for supervised clustering of data \cite{TDDataSegmentation}). The MBO scheme, its variants, and its extensions have attracted sustained interest in the context of each one of these applications.

\section{Preliminaries and Notation}

We will be concerned with isotropic interfacial energies defined on partitions of a domain $D$ into a maximum prescribed number $N\in\mathbb{N}$ sets.
Let
\[
\mathcal{I}_N = \left\{(i,j) \in \{1,\ldots,N\}\times\{1,\ldots,N\}: i \neq j\right\}
\]
denote pairs of distinct indices.
$D$ will typically be the $d$-dimensional annulus, i.e., a cube in $\R^d$ with periodic boundary conditions. By a partition of $D$, we mean $N$ closed sets $\Sigma_1,\ldots,\Sigma_N \subseteq D$, called phases, that may intersect only through their boundaries:
\[
D = \bigcup_{i=1}^N \Sigma_i \quad \text{and} \quad \Sigma_i \cap \Sigma_j = \left(\partial \Sigma_i\right) \cap \left(\partial \Sigma_i\right) \quad \text{for} \quad (i,j)\in\mathcal{I}_N.
\]
We denote by $\Gamma_{i,j}$ the interface separating $\Sigma_i$ and $\Sigma_j$ which is given by
\[
\Gamma_{i,j} = \left(\partial \Sigma_i\right) \cap \left(\partial \Sigma_i\right).
\]

Let $dH^s$ denote the $s$-dimensional surface area element.
Variational models for microstructure evolution proposed by Mullins \cite{MullinsGrainBoundaryMotion} take the form of the following penalty on partitions of $D$:
\bq
\label{eq:energy_multiphase_anisotropic}
E(\mathbf{\Sigma},\sigma) = \sum_{(i,j) \in \mathcal{I}_N} \int_{\Gamma_{i,j}} \sigma_{i,j}(n_{i,j}(x)) \:dH^{d-1}(x)
\eq
where we write $\bm{\Sigma} = (\Sigma_1,\ldots,\Sigma_N)$ and
$n_{i,j}(x)$ denotes the unit normal on $\Gamma_{i,j}$ pointing into $\Sigma_j$.
The functions $\sigma_{i,j} : \mathbb{S}^{d-1} \to \mathbb{R}^+$ are known as {\em surface tensions} associated with the interfaces $\Gamma_{i,j}$.
They are continuous, even functions that need to satisfy further properties to ensure well posedness of the model.

In this paper we focus on the special \emph{isotropic} case of (\ref{eq:energy_multiphase_anisotropic}) where the surface tensions $\sigma_{i,j}$ are constant but possibly distinct.
The multiphase energy then reduces to
\bq\label{eq:energy_multiphase_isotropic}
E(\bm{\Sigma},\sigma) = \sum_{(i,j) \in \mathcal{I}_N}\sigma_{i,j}H^{d-1}(\Gamma_{i,j}).
\eq
It is convenient to set $\sigma_{i,i} = 0$ and think of $\sigma$ as a symmetric matrix with $0$ along the diagonal and positive entries throughout:
\[
\mathcal{S}_N = \left\{\sigma \in \R^{N\times N}: \sigma_{i,i} = 0 \text{ and } \sigma_{i,j} = \sigma_{j,i} > 0 \text{ for } (i,j) \in \mathcal{I}_N\right\}.
\]
It turns out that the following triangle inequality is necessary and sufficient for the model \eqref{eq:energy_multiphase_isotropic} to be well-posed \cite{braides1990functionate}:
\bq\label{eq:triangle_inequality}
\sigma_{i,j} + \sigma_{j,k} \geq \sigma_{i,k} \quad \text{for any } i, j \text{ and } k.
\eq
We will therefore work mostly with the triangle-inequality-satisfying class of surface tensions:
\[
\mathcal{T}_N = \left\{\sigma \in \mathcal{S}_N: \sigma_{i,j} + \sigma_{j,k} \geq \sigma_{i,k} \text{ for any } i, j \text{ and } k\right\}.
\]

We will study approximations for $L^2$ gradient flow of energies \eqref{eq:energy_multiphase_isotropic} with special interest in the two-dimensional and three-dimensional cases. The normal speed under this flow is given by
\bq\label{normalspeed}
v_\perp(x) = \mu_{i,j} \sigma_{i,j} \kappa_{i,j}(x),
\eq
where $\kappa_{i,j}$ denotes the mean curvature of $\Gamma_{i,j}$. The constants $\mu_{i,j}$ are the \emph{mobilities} associated with the interfaces $\Gamma_{i,j}$. They are positive but otherwise can be chosen arbitrarily. In addition,  a condition known as the \emph{Herring angle condition} \cite{HerringAngleCondition} holds along triple junctions. At a junction formed by the meeting of the three phases $\Sigma_i$, $\Sigma_j$ and $\Sigma_k$, this condition reads
\bq\label{youngslaw}
\sigma_{i,j} n_{i,j} + \sigma_{j,k} n_{j,k} + \sigma_{k,i} n_{k,i} = 0.
\eq
Also known as Young's law in the isotropic setting considered here, this condition determines the opening angles $\theta_i$, $\theta_j$ and $\theta_k$ of the three phases $\Sigma_i$, $\Sigma_j$ and $\Sigma_k$, respectively, in terms of the surface tensions:
\bq\label{angles_relationship}
\frac{\sin(\theta_i)}{\sigma_{j,k}} = \frac{\sin(\theta_j)}{\sigma_{i,k}} = \frac{\sin(\theta_k)}{\sigma_{i,j}}.
\eq

\section{Previous work}\label{sec:previous_work}

In \cite{SelimFelix}, a variational formulation for the original MBO scheme (Algorithm \ref{alg:MBO}) was given. In particular,
it was shown that the following functional defined on sets, with kernel $K$ chosen to be the Gaussian $G$, which had previously been established \cite{AlbertiBellettini,MirandaShortTime} to be a non-local approximation to (isotropic) perimeter, is dissipated by the MBO scheme at every step, regardless of time step size:
\bq\label{eq:lyapunov_twophase}
E_{\sqrt{\delta t}}(\Sigma,K_{\sqrt{\delta t}}) = \frac{1}{\sqrt{\delta t}} \int_{\Sigma^c} K_{\sqrt{\delta t}}\ast \bm{1}_\Sigma \dx,
\eq
where for convenience we write
\begin{equation}
\label{eq:kernelnotation}
K_\epsilon(x) = \frac{1}{\epsilon^d}K\left(\frac{x}{\epsilon}\right).
\end{equation}

Thus, \eqref{eq:lyapunov_twophase} is a Lyapunov functional for Algorithm \ref{alg:MBO}, establishing its unconditional gradient stability. The next proposition from \cite{SelimFelix} illustrates this fact while also underlining the significance of $\widehat{K}$:
\begin{proposition}[from \cite{SelimFelix}]
Let $K$ satisfy
\bq\label{eq:kernel_standard_assumptions}
K(x) \in L^1(\Rd), \quad xK(x) \in L^1(\Rd), \quad \text{and} \quad K(x) = K(-x),
\eq
together with
\bq\label{eq:kernel_positive_mass}
\int_{\Rd} K(x)\dx > 0.
\eq
If $\widehat{K} \geq 0$, Algorithm \ref{alg:MBO} is unconditionally stable: each time step dissipates the energy \eqref{eq:lyapunov_twophase}, regardless of the time step size.
\end{proposition}

Moreover, in \cite{SelimFelix} the following, minimizing movements \cite{AlmgrenCurvatureDrivenFlows,LuckhausImplicitTime} interpretation involving \eqref{eq:lyapunov_twophase} for Algorithm \ref{alg:MBO} was given:
\bq\label{eq:minimizing_movement_twophase}
\Sigma^{k+1} = \argmin_{\Sigma} E_{\sqrt{\delta t}}(\Sigma,K_{\sqrt{\delta t}}) + \frac{1}{\sqrt{\delta t}} \int \left(\bm{1}_\Sigma-\bm{1}_{\Sigma_k}\right) K_{\sqrt{\delta t}} \ast \left(\bm{1}_\Sigma-\bm{1}_{\Sigma_k}\right) \dx,
\eq
where the kernel $K$ was again taken to be $G$. This variational formulation is then extended to the multiphase energy \eqref{eq:energy_multiphase_isotropic} where the surface tensions $\sigma_{i,j}$ are constant but possibly distinct. In this case the Lyapunov functional becomes
\bq\label{eq:lyapunov_multiphase_EO}
E_{\sqrt{\delta t}}(\bm{\Sigma},K_{\sqrt{\delta t}}) = \frac{1}{\sqrt{\delta t}} \sum_{(i,j)\in\mathcal{I}_N} \sigma_{i,j} \int_{\Sigma_j} K_{\sqrt{\delta t}} \ast \bm{1}_{\Sigma_i} \dx.
\eq
We also consider a relaxation of \eqref{eq:lyapunov_multiphase_EO}
\bq\label{eq:lyapunov_multiphase_EO_relaxation}
E_{\sqrt{\delta t}}(\bm{u},K_{\sqrt{\delta t}}) = \frac{1}{\sqrt{\delta t}} \sum_{(i,j)\in\mathcal{I}_N} \sigma_{i,j} \int_D u_j K_{\sqrt{\delta t}} \ast u_i \dx
\eq
over the following convex set of functions satisfying a box constraint:
\bq\label{eq:box_constraint}
\mathcal{K} = \left\{\bm{u} \in L^1\left(D,[0,1]^N\right): \sum_{i=1}^N u_i(x) = 1 \text{ a.e. } x \in D\right\}.
\eq
There is a corresponding minimizing movements scheme that can be derived from \eqref{eq:lyapunov_multiphase_EO_relaxation} that leads to the extension of threshold dynamics to the constant but possibly unequal surface tension multiphase energy \eqref{eq:energy_multiphase_isotropic} given in Algorithm \ref{alg:EO}.
\boxalgorithm{\caption{(in \cite{SelimFelix})}\label{alg:EO}}{Given the initial partition $\Sigma_1^0, \ldots, \Sigma_N^0$, to obtain the partition $\Sigma_1^{k+1}, \ldots, \Sigma_N^{k+1}$ at time step $t = (\delta t)(k+1)$ from the partition $\Sigma_1^k, \ldots, \Sigma_N^k$ at time $t = (\delta t)k$:
\begin{enumerate}
	\item Convolution step:
\[
\phi_i^k = K_{\sqrt{\delta t}} \ast \left(\sum_{j=1}^N \sigma_{i,j} \bm{1}_{\Sigma_j^K}\right) \quad i = 1,\ldots,N,
\]
where $K$ is the Gaussian \eqref{eq:Gaussian}.
	\item Thresholding step:
\[
\Sigma_i^{k+1} = \left\{x: \phi_i^k(x) < \min_{j\neq i} \phi_j^k(x)\right\}.
\]
\end{enumerate}
}

This algorithm is however restricted to very specific mobilities: $\mu_{i,j} = \frac{1}{\sigma_{i,j}}$. In \cite{SelimFelix}, a modified algorithm (see Algorithm \ref{alg:EO_retardation}) is proposed which allows for general mobilities by introducing {\em retardation terms}. Both algorithms are shown to be unconditionally gradient stable when the surface tension matrix $\sigma$ is \emph{conditionally negative semi-definite}:
\[
\left\{\sigma \in \mathcal{S}^N: \sum_{i,j=1}^N \sigma_{i,j} \xi_i \xi_j \leq 0 \text{ whenever } \sum_{i=1}^N \xi_i = 0\right\}.
\]
This corresponds to the matrices that are negative semi-definite as quadratic forms on $(1,\ldots,1)^\perp$.
\begin{proposition}[from \cite{SelimFelix}] \label{prop:SelimFelixStability}
Let the surface tension matrix $\sigma \in \mathcal{S}_N$ be conditionally negative semi-definite. Then Algorithms \ref{alg:EO} and \ref{alg:EO_retardation} are unconditionally stable: each time step dissipates energies \eqref{eq:lyapunov_multiphase_EO} and \eqref{eq:lyapunov_multiphase_EO_relaxation}.
\end{proposition}

\boxalgorithm{\caption{(in \cite{SelimFelix})}\label{alg:EO_retardation}}{Given the initial partition $\Sigma_1^0, \ldots, \Sigma_N^0$ with $\Sigma_i^0 = \left\{x:\psi_i^0(x) > 0\right\}$, to obtain the partition $\Sigma_1^{k+1}, \ldots, \Sigma_N^{k+1}$ at time step $t = (\delta t)(k+1)$ from the partition $\Sigma_1^k, \ldots, \Sigma_N^k$ at time $t = (\delta t)k$:
\begin{enumerate}
	\item Form the convolutions:
\[
\phi_i^k = K_{\sqrt{\delta t}} \ast \left(\sum_{j=1}^N \sigma_{i,j} \bm{1}_{\Sigma_j^K}\right) \quad i = 1,\ldots,N,
\]
where $K$ is the Gaussian \eqref{eq:Gaussian}.
	\item Form the retardation functions:
\[
R_i^k = \max_{l\neq i} \frac{\sqrt{\delta t}}{2}(1-\mu_{i,j}\sigma_{i,j}) \psi_j^k.
\]
	\item Form the comparison functions
\[
\psi_i^{k+1} = \left(\min_{l\neq i} \phi_l^k + \frac{1}{\sqrt{\delta t}} R_l^k\right) - \phi_i^k + \frac{1}{\sqrt{\delta t}} R_i^k.
\]
	\item Threshold the comparison functions $\psi_i^{k+1}$:
\[
\Sigma_i^{k+1} = \left\{x: \psi_i^k(x) > 0\right\}.
\]
\end{enumerate}
}

Unconditional gradient stability is a desirable property in threshold dynamic algorithms.
Given their minimizing movements formulation, requiring the $\Gamma$-convergence of the associated energies \eqref{eq:lyapunov_multiphase_EO_relaxation} also appears to be of fundamental importance.
We will highlight this point in \autoref{sec:nonconvergence} where we exhibit a counter example. We recall the result from \cite{SelimFelix} that establishes the $\Gamma$-convergence of Algorithm \ref{alg:EO}. In \cite{SelimFelix} the result is presented for general positive convolution kernels but here we focus on the case where $K$ is taken to be $G$. Let
\[
BV_\mathbb{B} = \left\{\bm{u} \in \mathcal{K}: u_i(x) \in \{0,1\} \text{ and } u_i \in BV(D) \text{ for } i \in \{1,2,\ldots,N\}\right\},
\]
and for $\bm{u}\in\mathcal{K}$ define the energy
\bq\label{eq:energy_multiphase_EO_relaxation}
E(\bm{u},\sigma) = \begin{cases}
\displaystyle \frac{1}{\sqrt{\pi}}\sum_{(i,j)\in\mathcal{I}_N} \sigma_{i,j} \int_D \nabla u_i + \nabla u_j - \nabla (u_i + u_j)	& \text{if } \bm{u} \in BV_\mathbb{B},\\
+\infty				& \text{otherwise},
\end{cases}
\eq
which is the formulation of multiphase energy \eqref{eq:energy_multiphase_isotropic} in the setting of functions of bounded variation.

\begin{theorem}[from \cite{SelimFelix}]\label{thm:Gamma_convergence_EO}
Assume that $\sigma \in \mathcal{T}_N$. Then, as $\epsilon\to 0$, the Lyapunov functionals $E_\epsilon(\cdot,G_\epsilon)$ given in \eqref{eq:lyapunov_multiphase_EO_relaxation} $\Gamma$-converge in the $L^1$ topology over $\mathcal{K}$ to the energy $E(\cdot,\sigma)$ given in \eqref{eq:energy_multiphase_EO_relaxation}. Furthermore, if for some sequence $\bm{u}_\epsilon$ we have $\sup_{\epsilon>0}E_\epsilon(\bm{u}_\epsilon,G_\epsilon) < \infty$, then $\bm{u}_\epsilon$ is precompact in $L^1(D)$ and the set of accumulation points is contained in $BV_\mathbb{B}(D)$.
\end{theorem}

Beyond the convergence of energies, recent work of Laux and Otto \cite{LauxOttoConvergenveMultiphase,LauxOttoBrakke} established conditional convergence of dynamics generated by Algorithm 2 to e.g. the multiphase version of a suitable weak formulation of mean curvature motion from \cite{LuckhausImplicitTime}.

When kernels more general than the Gaussian are used, the minimizing movements formulation \eqref{eq:minimizing_movement_twophase} allows to easily identify the corresponding possibly normal dependent surface tension and a normal dependent mobility factor associated with it.
In \cite{SelimElseyTD}, the following expressions for the surface tension $\sigma_K$ and mobility $\mu_K$ associated with a given kernel $K$ in terms of the Fourier transform $\widehat{K}$ are provided:
\bq\label{eq:identities_sigma_mu}
\sigma_K(n) = -\frac{1}{2\pi}\int_\R \frac{\widehat{K}(n\xi)-\widehat{K}(0)}{\xi^2}\:d\xi \quad \text{and} \quad \mu_K(n) = \pi\left(\int_\R \widehat{K}(n\xi)\:d\xi\right)^{-1}.
\eq
Here, we use the following definition of the Fourier transform on $\R^d$:
\[
\widehat{f}(\xi) = \int_{\Rd} f(x) e^{-ix\cdot\xi}\:d\xi \text{ so that } f(x) = \frac{1}{(2\pi)^d} \int_{\Rd} \widehat{f}(\xi)e^{i\xi\cdot x}\:d\xi,
\]
e.g. $f$ in Schwartz class. In addition, it is worth recalling the following fact from \cite{SelimElseyTD}.
\begin{proposition}[from \cite{SelimElseyTD}]
Let $\Sigma$ be a compact subset of $\Rd$ with smooth boundary. Let $K:\Rd \to \R$ be a kernel satisfying \eqref{eq:kernel_standard_assumptions}. Then
\[
\lim_{\delta t \to 0^+} E_{\sqrt{\delta t}}(\Sigma,K_{\sqrt{\delta t}}) = \int_{\partial \Sigma} \sigma_K(n(x)) \:dH^{d-1}(x)
\]
where the surface tension $\sigma_K:\Rd\to\R^+$ is defined as
\[
\sigma_K(n) = \frac{1}{2} \int_{\Rd} \abs{n \cdot x} K(x) \dx.
\]
\end{proposition}

Formulas \eqref{eq:identities_sigma_mu} that express the surface tension and mobility of a kernel have been previously used in \cite{SelimMattZahng} to design convolution kernels for a given desired pair of {\em anisotropic} surface tension and mobility. In that setting, the analogue of the Lyapunov functional \eqref{eq:lyapunov_multiphase_EO}, i.e., the non-local approximation to \eqref{eq:energy_multiphase_anisotropic}, is given by
\bq\label{eq:lyapunov_multiphase_EE}
E_{\sqrt{\delta t}}(\bm{\Sigma},\bm{K}_{\sqrt{\delta t}}) = \frac{1}{\sqrt{\delta t}} \sum_{(i,j)\in\mathcal{I}_N} \int_{\Sigma_j} \left(K_{i,j}\right)_{\sqrt{\delta t}} \ast \bm{1}_{\Sigma_i} \dx
\eq
where each component $K_{i,j}$ of the collection of kernels $\bm{K}$ satisfies
\bq
\label{eq:identities_sigma_mu_classic}
\frac{1}{2} \int_{\Rd} \abs{n\cdot x} K_{i,j}(x) \dx = \sigma_{i,j}(n)  \mbox{ and } \int_{n^\perp} K_{i,j}(x) \, dH^{d-1}(x) = \mu_{i,j}^{-1}(n),
\eq
while its relaxation given by
\bq\label{eq:lyapunov_multiphase_EE_relaxation}
E_{\sqrt{\delta t}}(\bm{u},\bm{K}_{\sqrt{\delta t}}) = \frac{1}{\sqrt{\delta t}} \sum_{(i,j)\in\mathcal{I}_N} \int_{\Rd} u_j \left(K_{i,j}\right)_{\sqrt{\delta t}} \ast u_i \dx.
\eq

In this paper, we use the same formulas \eqref{eq:identities_sigma_mu} in the {\em isotropic} setting to give a simpler version of Algorithm \ref{alg:EO_retardation} that does not require retardation terms, and which therefore stays truer to the spirit of the original MBO Algorithm \ref{alg:MBO}.
More specifically, the contributions of the paper can be summarized as follows:

\begin{enumerate}
	\item Given $N\choose 2$ surface tension and mobility pairs $(\sigma_{i,j} , \mu_{i,j})$, we construct $N\choose 2$ corresponding convolution kernels of the form
	$$ K_{i,j}(x) = a_{i,j} G_{\sqrt{\alpha}} + b_{i,j} G_{\sqrt{\beta}} $$
	where $a_{i,j} ,b_{i,j} >0$ and $\alpha \neq \beta$ are chosen to bake the desired $\sigma_{i,j}$ and $\mu_{i,j}$ into $K_{i,j}$.
	An essential novelty is that all required convolutions in the resulting threshold dynamics scheme can be obtained from convolutions with just two types of kernels, $G_{\sqrt{\alpha}}$ and $G_{\sqrt{\beta}}$, which makes for a particularly practical and efficient algorithm.
	\item We discuss conditions on $(\sigma_{i,j},\mu_{i,j})$, as well as $\alpha$, $\beta$, that ensure $\Gamma$-convergence of the corresponding non-local multiphase energy \eqref{eq:lyapunov_multiphase_EE_relaxation} to the corresponding sharp interface limit \eqref{eq:energy_multiphase_EO_relaxation}, and the unconditional gradient stability of the resulting thresholding scheme.
	It turns out that the permissible $(\sigma_{i,j},\mu_{i,j})$ pairs include {\em Read-Shockley surface tensions} \cite{ReadShockley} and {\em equal mobilities} $\mu_{i,j}=1$ which, unlike the unusual mobilities $\mu_{i,j} = \sigma_{i,j}^{-1}$ of Algorithm 2, are a very common choice in materials science literature.
	\item Although the convolution kernels in our construction are all positive and radially symmetric so that $\Gamma$-convergence of the two-phase energy (\ref{eq:lyapunov_twophase}) follows from prior work \cite{SelimFelix}, and even the older viscosity solutions approach of \cite{IshiiTDPropagatingFronts} applies and implies convergence of any of the two-phase flows, we exhibit choices of conditionally negative semi-definite $\sigma\in \mathcal{T}_N$ and $\mu\geq 0$ that fall outside our conditions for which the algorithm fails with the proposed kernel construction:
	the dynamics generated appears to converge to an unexpected limit.
	A short calculation shows that it is in fact $\Gamma$-convergence of the non-local multiphase energy \eqref{eq:lyapunov_multiphase_EE_relaxation} to the advertised limit \eqref{eq:energy_multiphase_EO_relaxation} that fails, and in this case, leads to failure of convergence of the dynamics as well.
\end{enumerate}

\section{The new algorithm}\label{sec:algorithm}

In this section, we derive the new, simplified version of Algorithm \ref{alg:EO_retardation} that dispenses with the retardation terms and still achieves a wide variety of mobilities, including the very important case of constant mobilities.
We then discuss its unconditional gradient stability, and the $\Gamma$-convergence of its associated non-local energies.
\subsection{Construction of the convolution kernels}

We begin with the simplest setting of two-phases: Given a target surface tension $\sigma_*$ and mobility $\mu_*$, we look for a kernel of the form
\[
K = a G_{\sqrt{\alpha}} + b G_{\sqrt{\beta}},
\]
where $\alpha > \beta > 0$ are fixed, and $G_{\sqrt{\alpha}}$ denotes the the Gaussian kernel given by
\[
G_{\sqrt{\alpha}}(x) = \frac{1}{4\pi\alpha}e^{-\frac{\abs{x}^2}{4\alpha}},
\]
and so $\widehat{G_{\sqrt{\alpha}}}(x) = e^{-\alpha\abs{x}^2}$.
Our goal is to choose $a,b>0$ so that $K$ has the desired $\sigma_*$ and $\mu_*$ as its surface tension and mobility via formulas \eqref{eq:identities_sigma_mu}.
Moreover, we would like to ensure $K>0$ -- a crucial property for the viscosity solutions approach \cite{IshiiTDPropagatingFronts} for two-phase flow, and very convenient for the variational formulation of \cite{SelimFelix} in establishing $\Gamma$-convergence of the corresponding energies.

We focus on the two-dimensional setting for convenience; the statements, formulas, and algorithms below adapt easily to arbitrary dimensions.

\begin{proposition}\label{prop:Kernel_construction}
Let $\alpha > \beta > 0$. Given $\sigma_*,\mu_* \in \R^+$, the convolution kernel $K = a G_{\sqrt{\alpha}} + b G_{\sqrt{\beta}}$, with
\[a = \frac{\sqrt{\pi}\sqrt{\alpha}}{\alpha-\beta}\left(\sigma_*-\beta\mu_*^{-1}\right) \quad \text{and} \quad b = \frac{\sqrt{\pi}\sqrt{\beta}}{\alpha-\beta}\left(-\sigma_*+\alpha\mu_*^{-1}\right)
\]
is such that $\sigma_K = \sigma_*$ and mobility $\mu_K = \mu_*$. Moreover, $K$ is positive if
\[
\alpha \geq \sigma_*\mu_* \quad \text{and} \quad \beta \leq \sigma_*\mu_*.
\]
\end{proposition}
\begin{proof}
We start by observing that for any $n \in \mathcal{S}^1$
\[
\sigma_{G_{\sqrt{\alpha}}}(n) = \frac{\sqrt{\alpha}}{\sqrt{\pi}} \quad \text{and} \quad \mu_{G_{\sqrt{\alpha}}}(n) = \sqrt{\pi}\sqrt{\alpha}.
\]
This follows from the formulas \eqref{eq:identities_sigma_mu} which also allows us to deduce that
\[
\begin{cases}
\sigma_{K} & = a \sigma_{G_{\sqrt{\alpha}}} + b_* \sigma_{G_{\sqrt{\beta}}},\\
\mu_{K}^{-1} & = a \mu_{G_{\sqrt{\alpha}}}^{-1} + b \mu_{G_{\sqrt{\beta}}}^{-1}.
\end{cases}
\]
Thus $(a,b)$ is the solution of the linear system
\[
\begin{cases}
\sigma_*	& = a \frac{\sqrt{\alpha}}{\sqrt{\pi}} + b \frac{\sqrt{\beta}}{\sqrt{\pi}},\\
\mu_*^{-1}	& = a \frac{1}{\sqrt{\pi}\sqrt{\alpha}} + b \frac{1}{\sqrt{\pi}\sqrt{\beta}},
\end{cases}
\]
and is given by
\[
\begin{cases}
a = \frac{\sqrt{\pi}\sqrt{\alpha}}{\alpha-\beta}\left(\sigma_*-\beta\mu_*^{-1}\right),\\
b = \frac{\sqrt{\pi}\sqrt{\beta}}{\alpha-\beta}\left(-\sigma_*+\alpha\mu_*^{-1}\right).
\end{cases}
\]
The kernel $K$ is positive if $a,b \geq 0$. Since $\alpha>\beta$, we need
\[
\sigma_*-\beta\mu_*^{-1} \geq 0 \Longleftrightarrow \beta \leq \sigma_*\mu_*
\quad
\text{and}
\quad
-\sigma_*+\alpha\mu_*^{-1} \geq 0 \Longleftrightarrow \alpha \geq \sigma_*\mu_*
\]
as desired.
\end{proof}

Of course, in the two-phase setting, the evolution generated by each such radially symmetric kernel is simply a constant multiple of mean curvature motion; only the product $\sigma_*\mu_*$ matters.
The individual values of $\sigma_*$ and $\mu_*$ become relevant in the multiphase setting, where surface tensions determine the junction angle conditions \eqref{angles_relationship}.
We thus now turn to the setting of $N$ phases, where
for each interface $\Gamma_{i,j}$ in the network, we are given a prescribed surface tension $\sigma_{i,j} \in \R^+$ and mobility $\mu_{i,j} \in \R^+$.
Using the construction of Proposition \ref{prop:Kernel_construction}, we define the kernel associated with interface $\Gamma_{i,j}$ as $K_{i,j} = a_{i,j} G_{\sqrt{\alpha}} + b_{i,j} G_{\sqrt{\beta}}$ with
\bq\label{eq:coefficients_kernels}
a_{i,j} = \frac{\sqrt{\pi}\sqrt{\alpha}}{\alpha-\beta}\left(\sigma_{i,j}-\beta\mu_{i,j}^{-1}\right) \quad \text{and} \quad b_{i,j} = \frac{\sqrt{\pi}\sqrt{\beta}}{\alpha-\beta}\left(-\sigma_{i,j}+\alpha\mu_{i,j}^{-1}\right).
\eq
Then, all the kernels $K_{i,j}$ are positive provided that
\bq\label{eq:choice_alpha_beta_positive}
\alpha \geq \max \sigma_{i,j}\mu_{i,j} \quad \text{and} \quad \beta \leq \min \sigma_{i,j}\mu_{i,j}.
\eq
This leads to the following natural algorithm to simulate the dynamics \eqref{normalspeed} under the constraint \eqref{angles_relationship}.

\boxalgorithm{\caption{}\label{alg:ES}}{Given the initial partial $\Sigma_1^0, \ldots, \Sigma_N^0$ with $\Sigma_i^0 = \left\{x:u_i^0(x) > 0\right\}$, to obtain the partition $\Sigma_1^{k+1}, \ldots, \Sigma_N^{k+1}$ at time step $t = (\delta t)(k+1)$ from the partition $\Sigma_1^k, \ldots, \Sigma_N^k$ at time $t = (\delta t)k$:
\begin{enumerate}
	\item \label{alg:ES_convolution_step}Form the convolutions:
\[
\phi^k_{1,i} = G_{\sqrt{\alpha\delta t}} * \mathbf{1}_{\Sigma_i^k} \mbox{ and } 
\phi^k_{2,i} = G_{\sqrt{\beta\delta t}} * \mathbf{1}_{\Sigma_i^k}.
\]
	\item \label{alg:ES_comparison}Form the comparison functions:
\begin{equation*}
\psi^k_i = \sum_{j \not= i} a_{i,j} \phi_{1,j}^k + b_{i,j} \phi_{2,j}^k
\end{equation*}
where $a_{i,j}$ and $b_{i,j}$ are given by \eqref{eq:coefficients_kernels}.
	\item \label{alg:ES_thresholding_step}Threshold the comparison functions:
\[
\Sigma_i^{k+1} = \left\{x: \psi_i^k(x) < \min_{j\neq i} \psi_j^k(x)\right\}.
\]
\end{enumerate}
}

\begin{remark}
The $2N$ convolutions required by Step (1) of the algorithm above can be quickly obtained using the acceleration utilized in \cite{ees} that groups distinct, well separated phases into families.
This way, even with hundreds of thousands of phases, the number of convolutions required per time step can be kept very low, depending not on $N$, but rather on the number of neighbors a typical phase in the network has.
\end{remark}

\begin{remark}
Algorithm \ref{alg:ES} is a special case of the algorithms from \cite{SelimElseyTD} and \cite{SelimMattZahng}.
Those references focus on fully anisotropic, multiphase setting, and require $N\choose 2$ possibly distinct convolution kernels, which renders the use of the accelleration method from \cite{ees} mentioned in the previous remark highly non-obvious.
The novelty of Algorithm \ref{alg:ES} is keeping the number of convolution kernels to just $2$.
\end{remark}


In \cite{SelimFelix}, it is shown that no wetting occurs when using Algorithms \ref{alg:EO} and \ref{alg:EO_retardation} as long as the surface tensions $\sigma_{i,j}$ satisfy the strict triangle inequality. By performing a similar analysis, we show the same result here. Indeed, let $p$ denote a point along one of the smooth surfaces $\Gamma_{i,j}$ away from any junction. Set $u_i^k = \bm{1}_{\Sigma_i^k}$. Then, by simply Taylor expanding we have, near $p$,
\[
G_{\sqrt{\delta t}} \ast u_i^k \approx G_{\sqrt{\delta t}} \ast u_j^k \approx \frac{a_{i,j}+b_{i,j}}{2},
\]
while $G_{\sqrt{\delta t}} \ast u^k_l$ for $l\notin \{i,j\}$ is exponentially small in $\delta t$ near $p$. Thus, near $p$, the coefficients $\psi^k_l$ given by \eqref{alg:ES_comparison} in Algorithm \ref{alg:ES} become
\[
\psi^k_l \approx \begin{cases}
(K_{l,i})_{\sqrt{\delta t}} \ast u^k_i + (K_{l,j})_{\sqrt{\delta t}} \ast u^k_j	& \text{if }l \notin \{i,j\},\\
(K_{i,j})_{\sqrt{\delta t}} \ast u^k_j						& \text{if }l = i,\\
(K_{i,j})_{\sqrt{\delta t}} \ast u^k_i						& \text{if }l = j,\\
\end{cases} = \begin{cases}
\frac{a_{l,i} + a_{l,j}}{2} + \frac{b_{l,i}+b_{l,j}}{2}	& \text{if }l \notin \{i,j\},\\
\frac{a_{i,j} +b_{i,j}}{2}				& \text{if }l = i,\\
\frac{a_{i,j} +b_{i,j}}{2}				& \text{if }l = j,\\
\end{cases}
\]
with an error that is exponentially small in $\delta t$. If the coefficients $a_{i,j}+b_{i,j}$ satisfy a strict triangle inequality, this implies
\[
\min\left\{\psi^k_i(x),\psi^k_j(x)\right\} < \psi^k_l(x) \quad \text{for all } l \notin \{i,j\}
\]
for $x$ near $p$. Hence, wetting does not occur: no new phase gets nucleated along $\Gamma_{i,j}$. Computations show that 
\[
a_{i,j}+b_{i,j}	= \frac{\sqrt{\pi}}{\sqrt{\alpha}+\sqrt{\beta}}\sigma_{i,j} + \frac{\sqrt{\alpha}\sqrt{\beta}}{\sqrt{\alpha}+\sqrt{\beta}}\mu_{i,j}^{-1}
\]
and so if the $\mu_{i,j}^{-1}$ satisfy the triangle inequality, the coefficients $a_{i,j}+b_{i,j}$ satisfy the strict inequality, as desired. On the other hand, if the $\mu_{i,j}^{-1}$ do not satisfy the triangle inequality, there exist $M,\epsilon>0$ such that
\[
\max_{i,j,k} \mu_{i,j}^{-1} - \left(\mu_{i,k}^{-1} + \mu_{k,j}^{-1}\right) = M \quad \text{and} \quad \max_{i,j,k} \sigma_{i,j} - \left(\sigma_{i,k} + \sigma_{k,j}\right) = -\epsilon.
\]
(recall we are assuming the $\sigma_{i,j}$ satisfy the strict triangle inequality). Now, due the definition of $\epsilon$ and $M$, computations show that $a_{i,j} +b_{i,j}$ satisfy the strict triangle inequality provided we choose $\alpha$ and $\beta$ such that
\bq\label{eq:choice_alpha_beta_triangle_inequality}
\sqrt{\alpha}\sqrt{\beta}M < \sqrt{\pi}\epsilon,
\eq
which is always possible by making $\beta$ sufficiently small.
Notice that the choice of $\beta$ depends only on the speed at which each the interface moves.
Indeed, $\beta$ remains unchanged  upon scaling the surface tensions and mobilities provided $\sigma_{i,j}\mu_{i,j}$ remains constant (i.e., the interfaces still move at the same speed).

\subsection{Stability and convergence}\label{sec:stability}

We begin by showing that Algorithm \ref{alg:ES} is unconditionally gradient stable under some mild assumptions. Then we focus on the specific case of the Read \& Shockley model \cite{ReadShockley} and discuss as well the $\Gamma$-convergence of associated non-local energies.

\begin{proposition}\label{prop:stability}
Let the surface tensions matrix $\sigma \in \mathcal{T}_N$ and the matrix of reciprocal mobilities $\frac{1}{\mu} \in \mathcal{S}_N$ be conditionally negative definite.
Choose $\alpha$ and $\beta$ to satisfy
\bq\label{eq:choice_alpha_beta_cnsd}
\alpha \geq \frac{\min\limits_{i = 1,\ldots,N-1} s_i}{\max\limits_{i = 1,\ldots,N-1} m_i} \quad \text{and} \quad
\beta \leq \frac{\max\limits_{i = 1,\ldots,N-1} s_i}{\min\limits_{i = 1,\ldots,N-1} m_i},
\eq
where $s_i$ and $m_i$ are the nonzero eigenvalues of $J\sigma J$ and $J\frac{1}{\mu}J$, respectively, with $J = I-\frac{1}{N}e e^T$ denoting the orthogonal projection to $e^\perp = (1,\ldots,1)^\perp$.
Let the kernels $K_{i,j}$ be given by $K_{i,j} = a_{i,j} G_\alpha + b_{i,j} G_\beta$, where $a_{i,j}$ and $b_{i,j}$ are given by \eqref{eq:coefficients_kernels} and satisfy \eqref{eq:choice_alpha_beta_positive}.
Then, Algorithm \ref{alg:ES} is unconditionally gradient stable:
Each time step dissipates the non-local energies \eqref{eq:lyapunov_multiphase_EE} and \eqref{eq:lyapunov_multiphase_EE_relaxation}.
\end{proposition}

\begin{proof}
Examining the proof of Proposition 5.3 in \cite{SelimFelix}, it is sufficient to show that both $A = (a_{i,j})$ and $B = (a_{i,j})$ are conditionally negative semi-definite, where $a_{i,j}$ and $b_{i,j}$ are given according to \eqref{eq:coefficients_kernels}. This will follow from showing that
\[
\sup_{\substack{v\neq0\\v \in e^\perp}} \frac{v^TA v}{\abs{v}^2} < 0 \quad \text{and} \quad \sup_{\substack{v\neq0\\v \in e^\perp}} \frac{v^TB v}{\abs{v}^2} < 0.
\]
In addition, one can show that
\[
\sup_{\substack{v\neq0\\v \in e^\perp}} \frac{v^T\sigma v}{\abs{v}^2} = \max_{i = 1,\ldots,N-1} s_i \quad \text{and} \quad 
\inf_{\substack{v\neq0\\v \in e^\perp}} \frac{v^T\sigma v}{\abs{v}^2} = \min_{i = 1,\ldots,N-1} s_i
\]
and similarly for $\frac{1}{\mu}$. The proof then follows due to the choice of $\alpha$ and $\beta$.
\end{proof}

We now recall models for surface tensions and mobilities that are very commonly used in grain boundary motion simulations in the materials science literature: Read-Shockley surface tensions along with equal mobilities.
In \cite{ReadShockley}, Read and Shockley describe a model for the grain boundary formed between two planar grains with square lattices.
Their calculation shows that the surface tension of the boundary is a specific function of the misorientation angle between the two lattices under the assumption that this angle is small. Each grain is assigned an orientation angle: when restricted to a plane, the orientation of a square lattice is uniquely determined by the angle $\theta$ of a clockwise rotation about the origin that maps it back to the standard two-dimensional lattice $\Z^2$. Given the symmetries of the square lattice, we can take $\theta \in [-\frac{\pi}{4},\frac{\pi}{4}]$ with the two endpoints of the interval identified. 
The surface tension $\sigma_{i,j}$ of the interface between two grains with orientations $\theta_i$ and $\theta_j$ in the Read-Shockley model for two-dimensional crystallography has the form
\bq\label{eq:RS2Dmodel_surface_tension}
\sigma_{i,j} = \min_{k\in\Z} f\left(\Abs{\theta_i-\theta_j+k\frac{\pi}{2}}\right)
\eq
where $f:\R^+\to\R$ satisfies
\begin{enumerate}[label=(H\arabic*)]
	\item \label{H1} $f \in C([0,\infty)) \cap C^2((0,\infty))$ and $\lim_{\xi\to 0^+} \xi^2f'(\xi) = 0$,
	\item $f(0) = 0$ and $f(\xi) \geq 0$ for all $\xi$,
	\item $f'(\xi) \geq 0$ for all $\xi > 0$,
	\item \label{H4} $f''(\xi) \leq 0$ for all $\xi > 0$.
\end{enumerate}

We follow the extension of Read-Shockley model to three-dimensional crystallography given in \cite{HolmMisorientationAngles}.
The orientation of a grain with cubic lattice can be described (nonuniquely) by a matrix $g \in SO(3)$ that corresponds to the rotation required to obtain the lattice of the grain from the standard integer lattice $\Z^3$.
In turn, any matrix $g \in SO(3)$ can be described as a rotation by an angle $\theta \in [0,\pi]$ about an axis $v\in\mathcal{S}^2$. According to \cite{HolmMisorientationAngles}, the surface tension $\sigma_{i,j}$ of the interface $\Gamma_{i,j}$ depends only on the misorientation angle (and not on the axis) between the two grains $g_i$ and $g_j$ and can be defined as follows. Let $\bO$ denote the octahedral group (of symmetries of the cube in the three dimensions). Define the minimal angle of rotation $\theta_{\bO}(g)$ of $g \in SO(3)$ as
\[
\theta_{\bO}(g) = \min_{r\in\bO} \theta(rg).
\]
The misorientation angle $g_i$ and $g_j$ is defined to be
\bq\label{eq:RS3Dmodel_misorientation_angle}
\theta_{i,j} = \theta_{\bO}(g_ig_j^T),
\eq
and the corresponding surface tension $\sigma_{i,j}$ is given by
\bq\label{eq:RS3Dmodel_surface_tension}
\sigma_{i,j} = f(\theta_{i,j})
\eq
where $f:\R^+\to\R$ is given by
\bq\label{eq:RS3Dmodel_f}
f(\theta) = \begin{cases}
\frac{\theta}{\theta_*}\left(1-\log\left(\frac{\theta}{\theta_*}\right)\right)	& \text{if } \theta < \theta_*,\\
1	& \text{if } \theta \geq \theta_*
\end{cases}
\eq
as in \cite{HolmMisorientationAngles,ReadShockley}. Here $\theta_*$ is a critical misorientation value, known as the Brandon angle, that denotes the rotation angle beyond which the surface tension saturates. According to \cite{HolmMisorientationAngles}, it has ben experimentally determined to lie somewhere betweeen $10^\circ$ and $30^\circ$.

\begin{theorem}\label{thm:RSmodel}
For a network of grains in dimension $d\in\{2,3\}$ in which each grain has a distinct orientation, let the surface tensions $\sigma_{i,j}$ of the grain boundaries be given by the Read \& Shockley model \eqref{eq:RS2Dmodel_surface_tension} for $d=2$, or the extension \eqref{eq:RS3Dmodel_surface_tension} of the same to 3D crystallography with $\theta_* \leq \frac{\pi}{4} = 45^\circ$ for $d=3$. Let the mobilities of the boundaries be given by $\mu_{i,j} = 1$ and construct the kernels $K_{i,j} = a_{i,j} G_\alpha + b_{i,j} G_\beta$, where $a_{i,j}$ and $b_{i,j}$ are given by \eqref{eq:coefficients_kernels} and satisfy \eqref{eq:choice_alpha_beta_positive}.

Then:
\begin{enumerate}
\item \label{statement_thm:stability} Algorithm \ref{alg:ES} is unconditionally gradient stable, i.e.,
each time step dissipates the non-local energies \eqref{eq:lyapunov_multiphase_EE} and \eqref{eq:lyapunov_multiphase_EE_relaxation}, provided $\alpha$ and $\beta$ satisfy \eqref{eq:choice_alpha_beta_cnsd}.
\item \label{statement_thm:convergence} The non-local energy \eqref{eq:lyapunov_multiphase_EE_relaxation} converges to the sharp interface limit \eqref{eq:energy_multiphase_EO_relaxation} in the sense of $\Gamma$-convergence as $\delta t \to 0^+$.
\end{enumerate}
\end{theorem}

\begin{proof}
A close examination of Theorem 5.5 and 5.6 from \cite{SelimFelix} shows that the surface tension matrix $\sigma$ is in fact conditionally negative definite, provided that all grains in the network have distinct orientations. Thus \eqref{statement_thm:stability} follows by Proposition \ref{prop:stability}. On the other hand, a close examination of Proposition A.1 in \cite{SelimFelix} shows that the proof can be extended to the non-local energies \eqref{eq:lyapunov_multiphase_EE_relaxation} when the kernels $K_{i,j}$ are positive and satisfy a pointwise triangle inequality. That is precisely the case here due to condition \eqref{eq:choice_alpha_beta_positive} which guarantees the positiveness of the kernels and since both $\sigma,\frac{1}{\mu} \in \mathcal{T}_N$. Thus \eqref{statement_thm:convergence} follows as well.
\end{proof}

\begin{remark}
The results in Theorem \ref{thm:RSmodel} extend to reciprocal mobility matrices in $\mathcal{T}_N$ that are conditionally negative definite.
\end{remark}

\section{Non-convergence of nonlocal multiphase energies}\label{sec:nonconvergence}

In this section, we discuss whether Algorithm \ref{alg:ES} along with the proposed kernel construction can be used on more general surface tension and mobility pairs than allowed by Theorem \ref{thm:RSmodel}, for instance for the larger class of Proposition \ref{prop:SelimFelixStability}.
In particular, we show that even for perfectly reasonable choices of $(\sigma_{i,j},\mu_{i,j})$, the algorithm can fail to converge to the correct evolution.
The culprit turns out to be failure of $\Gamma$-convergence of the correspoding multiphase non-local energy to the expected limit \eqref{eq:energy_multiphase_EO_relaxation} where surface tensions are given by the simple formula \eqref{eq:identities_sigma_mu}.
More specifically, our example has the following features:
\begin{enumerate}
\item The prescribed surface tensions $\sigma_{i,j}$ satisfy the triangle inequality \eqref{eq:triangle_inequality}, so that multiphase model \eqref{eq:energy_multiphase_isotropic} is well-posed.
\item The prescribed surface tensions $\sigma_{i,j}$ and reciprocal mobilities $\mu_{i,j}^{-1}$ are conditionally negative semi-definite, so that Algorithm \ref{alg:ES} is unconditionally gradient stable, decreasing the multiphase non-local energy \eqref{eq:lyapunov_multiphase_EE} at every time step.
\item The corresponding kernels $K_{i,j}$ are positive, so that the non-local \emph{two-phase} energy \eqref{eq:lyapunov_twophase}
corresponding to each $K_{i,j}$ converges in the sense of $\Gamma$-convergence to
\[
E(\Sigma,\sigma_{i,j}) = \int_{\partial\Sigma} \sigma_{i,j}(n(x)) \:dH^{d-1}(x)
\]
with the desired surface tension $\sigma_{i,j}$ given by formula \eqref{eq:identities_sigma_mu} by construction of $K_{i,j}$.
\item However, the $\Gamma$-limit of the \emph{multiphase} non-local energies \eqref{eq:lyapunov_multiphase_EE_relaxation} utilizing the $K_{i,j}$ as convolution kernels is {\em not} the desired limit \eqref{eq:energy_multiphase_EO_relaxation} where the surface tensions are the prescribed $\sigma_{i,j}$.
\item In general, failure of $\Gamma$-convergence need not imply failure of gradient descent dynamics, since gradient flow -- a local optimization strategy -- may never find the dramatically energy reducing perturbation to the structure of the interface that is responsible for the failure of $\Gamma$-convergence.
However, our numerical experiments in \autoref{sec:numerics} show that Algorithm \ref{alg:ES} in fact finds the perturbation, which then leads to the failure of the dynamics it generates.
\end{enumerate}

\subsection{A simple counterexample to $\Gamma$-convergence}\label{sec:counterexample}

The example below shows that simply requiring each kernel $K_{i,j}$ to satisfy sufficient conditions (i.e. $K_{i,j}\geq 0$ and \eqref{eq:identities_sigma_mu_classic} with our construction) for $\Gamma$-convergence in the corresponding two-phase setting is not sufficient for $\Gamma$-convergence of the multiphase energy, even when the desired surface tensions $\sigma_{i,j}$ satisfy the triangle inequality and hence come from a well-posed sharp interface variational model.
It is described in the 1-dimensional setting for simplicity (but can of course be extended to any dimensions).

\begin{example}\label{ex:counterexample}
Let $\Omega = \R$ and consider its partition into $N=3$ phases, parameterized by $\epsilon > 0$, given by
\[
\Sigma_{1,\epsilon}	= \left\{x\in\Omega:x\geq\epsilon\right\}, \quad \Sigma_{2,\epsilon} = \left\{x\in\Omega:x\leq-\epsilon\right\}, \quad \Sigma_{3,\epsilon} = \left\{x\in\Omega:\abs{x}\leq\epsilon\right\}.
\]
Set
\[
u_{1,\epsilon} = \bm{1}_{\Sigma_{1,\epsilon}} = \bm{1}_{[\epsilon,\infty)}, \quad u_{2,\epsilon} = \bm{1}_{\Sigma_{2,\epsilon}} = \bm{1}_{(-\infty,-\epsilon]}, \quad u_{3,\epsilon} = \bm{1}_{\Sigma_{3,\epsilon}} = \bm{1}_{[-\epsilon,\epsilon]}.
\]
Define the kernels as follows:
\[
K_{1,2} = \bm{1}_{[-1,1]} \quad \text{and} \quad K_{1,3} = K_{2,3} = \delta_1\left(\bm{1}_{[-11,-9]}+\bm{1}_{[9,11]}\right) + \delta_2\bm{1}_{[-1,1]},
\]
where $\delta_1,\delta_2 > 0$ will be chosen later. Notice that the corresponding surface tensions are given by
\[
\sigma_{1,2} = \frac{1}{2}\int_\R \abs{x} K_{1,2}(x)\dx = \frac{1}{2} \quad \text{and} \quad \sigma_{1,3} = \sigma_{2,3} = \frac{1}{2}\int_\R \abs{x} K_{2,3}(x)\dx = 20 \delta_1 + \frac{\delta_2}{2}.
\]
while the mobilities are given by
\[
\mu_{1,2} = \frac{1}{2K_{1,2}(0)} = \frac{1}{2} \quad \text{and} \quad \mu_{1,3} = \mu_{2,3} = \frac{1}{2K_{2,3}(0)} = \frac{1}{2\delta_2}.
\]
Here, the non-local approximate energy is given by
\[
E_{\epsilon} (\bm{u}_\epsilon,\bm{K}_\epsilon) = \frac{1}{\epsilon} \sum_{(i,j) \in \mathcal{I}_N} \int_\R u_{j,\epsilon}(x) \left(K_{i,j}\right)_\epsilon \ast u_{i,\epsilon}(x) \dx = 16 \delta_1 + 2\delta_2,
\]
which follows from
\begin{align*}
	& \:\frac{1}{\epsilon} \int_\R u_{1,\epsilon}(x) \left(K_{1,3}\right)^\epsilon \ast u_{3,\epsilon}(x) \dx\\
=	& \:\frac{1}{\epsilon} \int_\R \int_\R \bm{1}_{[\epsilon,\infty)}(x)  \left(\delta_1 \bm{1}_{[9,11]}(h) + \delta_2 \bm{1}_{[-1,1]}(h)\right) \bm{1}_{[-\epsilon,\epsilon]}(x-\epsilon h) \:dh\dx\\
=	& \: 4\delta_1 + \frac{\delta_2}{2}
\end{align*}
(since the double integral in the second equality corresponds to the area of a parallelogram of base $2\epsilon$ and height 2 and a triangle of side $\epsilon$ and height $1$) and
\[
\frac{1}{\epsilon} \int_\R u_{1,\epsilon}(x) \left(K_{1,2}\right)^\epsilon \ast u_{2,\epsilon}(x) \dx = \frac{\delta}{\epsilon} \int_\R \int_\R \bm{1}_{[\epsilon,\infty)}(x)  \bm{1}_{[-1,1]}(h) \bm{1}_{(-\infty,-\epsilon]}(x-\epsilon h) \:dh\dx = 0.
\]

On the other hand, the limiting energy is given by
\[
E(\bm{u}_0,\sigma) = 2 \: \sigma_{1,2} = 1.
\]
Now, we observe that e.g. the choice $\delta_1 = \frac{1}{64}$ and $\delta_2 = 5/16$ guarantees that both the surface tension matrix $\sigma$ and reciprocal mobility matrix $\frac{1}{\mu}$ are conditionally negative semi-definite. Moreover, the surface tension matrix satisfies the triangle inequality. However $\Gamma$-convergence fails since
\[
\liminf_{\epsilon \to 0} E_{\epsilon} (\bm{u}_\epsilon,\bm{K}_\epsilon) < E(\bm{u}_0,\sigma).
\]
Notice that in this case the reciprocal mobilities $\mu_{i,j}^{-1}$ do not satisfy the triangle inequality.
\end{example}

\section{Numerical evidence}\label{sec:numerics}

We present a variety of numerical tests for Algorithm \ref{alg:ES}. We focus on classical numerical convergence studies for short-time evolution (during which topological changes do not take place) starting from an initial condition with triple junctions formed by the meeting of smooth curves. It is worth mentioning nonetheless that threshold dynamics methods shine when it comes to challenging configurations that involve topological changes. For all the examples considered below, both $\sigma$ and $\frac{1}{\mu}$ are conditionally negative definite. We choose $\alpha$ and $\beta$ as the smallest and largest constants, respectively that satisfy both \eqref{eq:choice_alpha_beta_positive} and \eqref{eq:choice_alpha_beta_cnsd}. This guarantees that the kernels are positive and that Algorithm \ref{alg:ES} dissipates the non-local energies \eqref{eq:lyapunov_multiphase_EE} and \eqref{eq:lyapunov_multiphase_EE_relaxation} at each time step.

\subsection{Comparisons with Exact Solutions}

We start by considering two examples for which the exact solutions of the threshold dynamics \eqref{normalspeed} and \eqref{angles_relationship} are well-known. These solutions are known as \emph{grim-reaper} solution \cite{GrimReaperSolution}: two of the interfaces are travelling waves moving with constant vertical speed, while the third remains a line segment.

\begin{example}\label{ex:grim_reaper_1}
We consider first the following symmetric case where the surface tension and mobility matrices are given by
\[
\sigma = \begin{pmatrix}
0 & \sqrt{2} & 1\\
\sqrt{2} & 0 & 1\\
1 & 1 & 0
\end{pmatrix} \quad \text{and} \quad
\mu = \begin{pmatrix}
0 & 1 & 1\\
1 & 0 & 1\\
1 & 1 & 0
\end{pmatrix}.
\]

The corresponding angles at the junctions are ($135^\circ,135^\circ,90^\circ)$. The two interfaces $\Gamma_{1,2}$ and $\Gamma_{1,3}$ are then graphs of functions $f_{1,2}(x,t):\left[0,\frac{1}{4}\right]\to \R$ and $f_{1,3}(x,t):\left[\frac{1}{4},\frac{1}{2}\right]\to \R$ that move by vertical translations:
\[
f_{1,2}(x,t) = \frac{7}{8} + \frac{1}{\pi}\log(\cos(\pi x)) - \pi t \quad \text{and} \quad f_{1,3}(x,t) = f_{1,3}\left(\frac{1}{2}-x,t\right).\\
\]
The interfaces satisfy the natural boundary condition of $90^\circ$ intersection with the boundary of the domain $[0,\frac{1}{2}]\times [0,\frac{1}{2}]$ (i.e. $\partial_x f_{1,3}(0,t) = 0$ and $\partial_x f_{2,3}(0,t) = 0$. Numerically, the initial configuration is extended evenly to $[0,1]\times [0,1]$ by reflection, which is then computed with periodic boundary conditions using Algorithm \ref{alg:ES}. The $L^\infty$ error between the computed and exact $f_{1,3}$ and $f_{2,3}$ at time $t = 0.1$ is shown in the Table \ref{table:ex_grim_reaper}. Figure \ref{fig:grim_reaper_1} shows the initial solution, the computed solution, and the exact solution in black, blue and red, respectively.
\end{example}

\begin{example}\label{ex:grim_reaper_2}
We consider as well an asymmetric grim-reaper solution. In this case, the surface tension and mobility matrices are given by
\[
\sigma = \begin{pmatrix}
0 & 1 & \frac{\sqrt{2}}{1+\sqrt{3}}\\
1 & 0 & \frac{2}{1+\sqrt{3}}\\
\frac{\sqrt{2}}{1+\sqrt{3}} & \frac{2}{1+\sqrt{3}} & 0
\end{pmatrix} \quad \text{and} \quad
\mu = \begin{pmatrix}
0 & \frac{1}{4\sqrt{2}} & 1\\
\frac{1}{4\sqrt{2}} & 0 & \frac{1}{4\sqrt{2}}\\
1 & \frac{1}{4\sqrt{2}} & 0
\end{pmatrix}.
\]
The corresponding angles at the junctions are ($135^\circ,150^\circ,75^\circ)$. The two interfaces $\Gamma_{1,2}$ and $\Gamma_{1,3}$ are then graphs of functions $f_{1,2}(x,t):\left[0,\frac{3}{8}\right]\to \R$ and $f_{1,3}(x,t):\left[\frac{3}{8},\frac{1}{2}\right]\to \R$ that move by vertical translations:
\begin{align*}
f_{1,2}(x,t) & = \frac{3}{2\pi}\log\left(\cos\left(\frac{2\pi}{3}x\right)\right) - \frac{2\sqrt{2}\pi}{3(1+\sqrt{3})}t,\\
f_{1,3}(x,t) & = \frac{3}{8\pi}\log\left(\frac{1}{2}\cos\left(\frac{4\pi(1-2x)}{3}\right)\right) - \frac{2\sqrt{2}\pi}{3(1+\sqrt{3})}t.\\
\end{align*}
Like in the previous example, the initial configuration is extended evenly to $[0,1]\times[0,1]$ with periodic boundary conditions and we use Algorithm \ref{alg:ES}. The $L^\infty$ error between the computed and exact $f_{1,3}$ and $f_{2,3}$ at time $t = 0.096$ is shown in the Table \ref{table:ex_grim_reaper}. Figure \ref{fig:grim_reaper_2} shows the initial solution, the computed solution, and the exact solution in black, blue and red, respectively.

\end{example}

\begin{table}[htp]
\centering
\footnotesize
\begin{tabular}{cccc}
\hline
 & \multicolumn{3}{l}{Errors and order, Example \ref{ex:grim_reaper_1}} \\ \cline{2-4}
$\#$ Time steps	& $\#$ Grid points	& $L^\infty$ error	& Conv. rate\\
\hline
50 & $128 \times 128$ & 0.0074 & - \\
100 & $256 \times 256$ & 0.0058 & 0.35 \\
200 & $512 \times 512$ & 0.0044 & 0.41 \\
400 & $1024 \times 1024$ & 0.0029 & 0.61 \\
800 & $2048 \times 2048$ & 0.0020 & 0.55 \\
1600 & $4096 \times 4096$ & 0.0014 & 0.52 \\
 & \multicolumn{3}{l}{Errors and order, Example \ref{ex:grim_reaper_2}} \\ \cline{2-4}
$\#$ Time steps	& $\#$ Grid points	& $L^\infty$ error	& Conv. rate\\
\hline
20 & $128 \times 128$ & 0.0239 & - \\
40 & $256 \times 256$ & 0.0183 & 0.39 \\
80 & $512 \times 512$ & 0.0133 & 0.46 \\
160 & $1024 \times 1024$ & 0.0094 & 0.49 \\
320 & $2048 \times 2048$ & 0.0060 & 0.66 \\
640 & $4096 \times 4096$ & 0.0036 & 0.74 \\
\end{tabular}
\caption{Errors in the $L^\infty$ norm for grim-reaper type examples.}
\label{table:ex_grim_reaper}
\end{table}

\begin{figure}[htp]
\centering
\subfigure[]{
\includegraphics[width=0.4\textwidth]{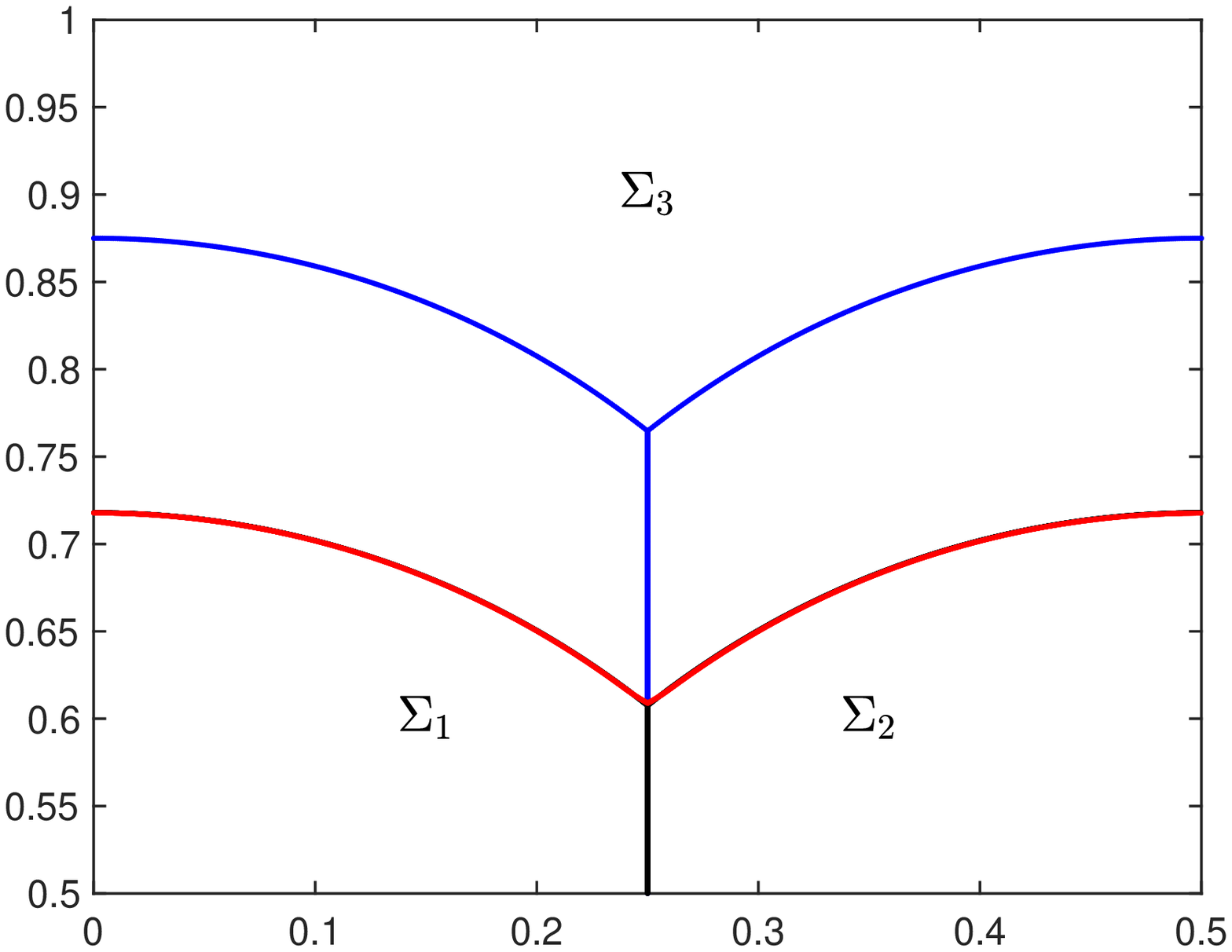}\label{fig:grim_reaper_1}}
\hspace{2ex} 
\subfigure[]{
\includegraphics[width=0.4\textwidth]{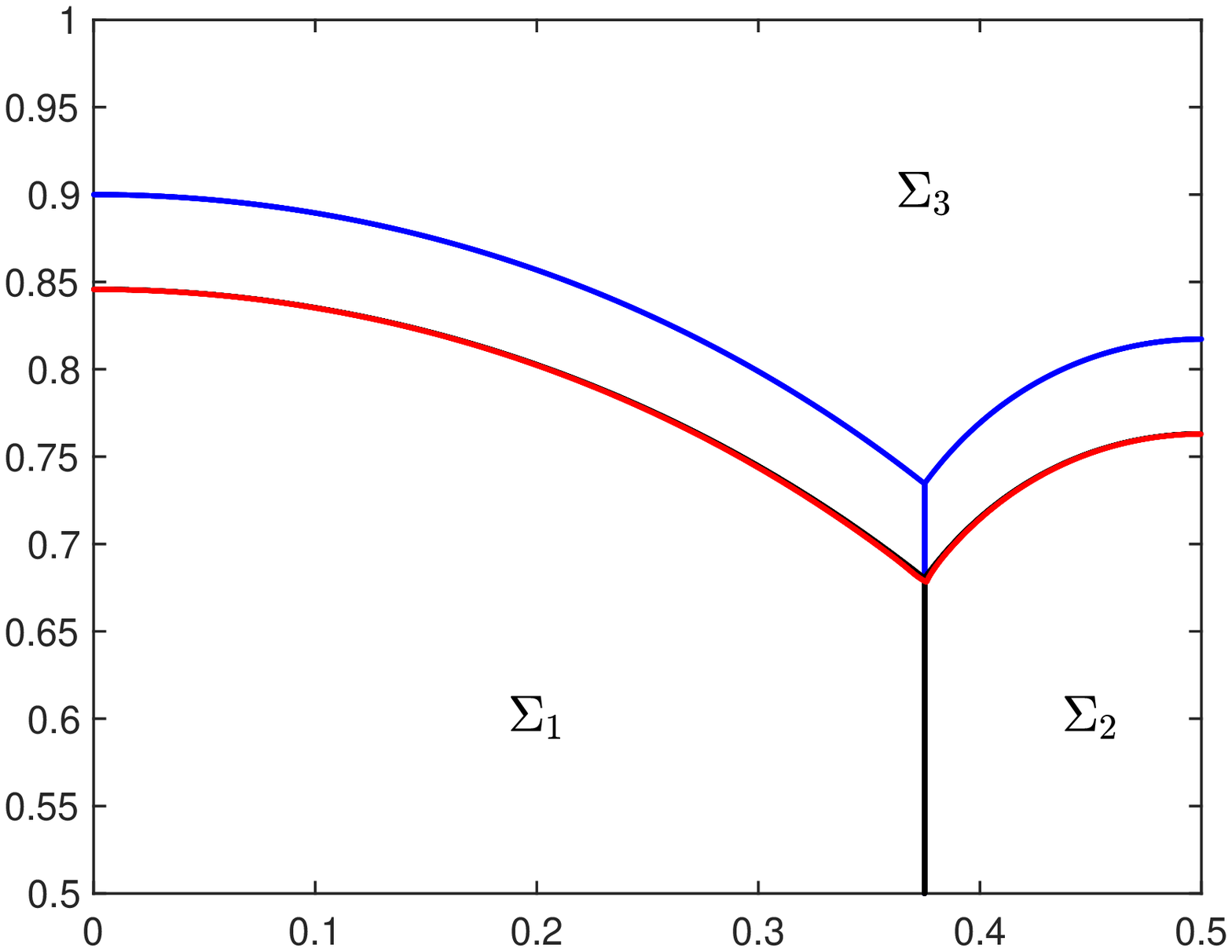}\label{fig:grim_reaper_2}}
\caption{Evolution of a three-phase grim-reaper like configuration where the blue curve shows the initial condition. Final configuration computed using threshold dynamics Algorithm \ref{alg:ES} (red), compared to the exact solution (black).}
\label{fig:grim_reaper}
\end{figure}

\subsection{Comparisons with Front Tracking}

In the absence of topological changes, and when starting from a smooth initial condition consisting only of triple junctions, a very appropriate and efficient algorithm for computing the curvature flow \eqref{normalspeed} under constraint \eqref{angles_relationship} is \emph{front tracking} (see e.g. \cite{FrontTracking}), especially in the plane.

\begin{example}\label{ex:front_tracking_1}
The initial condition in this example is shown in Figure \ref{fig:front_tracking_1} as the blue curve. It is evolved under dynamics \eqref{normalspeed} and \eqref{angles_relationship} with surface tension and mobility matrices are given by
\[
\sigma = \begin{pmatrix}
0 & 1 & 1\\
1 & 0 & \sqrt{2}\\
1 & \sqrt{2} & 0
\end{pmatrix} \quad \text{and} \quad
\mu = \begin{pmatrix}
0 & 1 & 1\\
1 & 0 & 1\\
1 & 1 & 0
\end{pmatrix}.
\]
The corresponding junction angles are $(\theta_1,\theta_2,\theta_3) = (90^\circ,135^\circ,135^\circ)$. The final configuration at time $t = 0.0107$, computed using Algorithm \ref{alg:ES} on a $4096 \times 4096$ grid, is shown as the red curve. The same configuration computed via front tracking is shown as the black curve. Table \ref{table:ex_front_tracking} shows the error as measured in the Hausdorff distance between the boundary $\partial \Sigma_1$ of phase $\Sigma_1$ computed using front tracking versus the proposed algorithm.
\end{example}

\begin{example}\label{ex:front_tracking_2}
The same initial condition as in Example \ref{ex:front_tracking_1} (blue curve in Figure \ref{fig:front_tracking_2}) was used for testing Algorithm \ref{alg:ES} but with different surface tensions:
\[
\sigma = \begin{pmatrix}
0 & \frac{5}{4} & \frac{3}{2}\\
\frac{5}{4} & 0 & 1\\
\frac{3}{2} & 1 & 0
\end{pmatrix}.
\]
The corresponding junction angles are $(\theta_1,\theta_2,\theta_3) \approx (138.6^\circ,97.18^\circ,124.2^\circ)$. The table below shows the error in phase $\Sigma_1$, once again as measured in the Hausdorff distance between the front tracking and the solution obtained from Algorithm \ref{alg:ES}.
\end{example}
\begin{table}[htp]
\centering
\footnotesize
\begin{tabular}{cccc}
\hline
 & \multicolumn{3}{l}{Errors and order, Example \ref{ex:front_tracking_1}} \\ \cline{2-4}
$\#$ Time steps	& $\#$ Grid points	& Hausdorff dist.	& Conv. rate\\
\hline
11 & $128 \times 128$ & 0.0149 & - \\
21 & $256 \times 256$ & 0.0075 & 0.98 \\
43 & $512 \times 512$ & 0.0059 & 0.35 \\
86 & $1024 \times 1024$ & 0.0041 & 0.52 \\
171 & $2048 \times 2048$ & 0.0027 & 0.58 \\
342 & $4096 \times 4096$ & 0.0021 & 0.40 \\
 & \multicolumn{3}{l}{Errors and order, Example \ref{ex:front_tracking_2}} \\ \cline{2-4}
$\#$ Time steps	& $\#$ Grid points	& Hausdorff dist.	& Conv. rate\\
\hline
11 & $128 \times 128$ & 0.0129 & - \\
21 & $256 \times 256$ & 0.0054 & 1.25 \\
43 & $512 \times 512$ & 0.0059 & -0.13 \\
86 & $1024 \times 1024$ & 0.0045 & 0.39 \\
171 & $2048 \times 2048$ & 0.0028 & 0.67 \\
342 & $4096 \times 4096$ & 0.0021 & 0.45 \\
\end{tabular}
\caption{Errors in the Hausdorff distance for the front tracking examples.}
\label{table:ex_front_tracking}
\end{table}

\begin{figure}[htp]
\centering
\subfigure[]{
\includegraphics[width=0.4\textwidth]{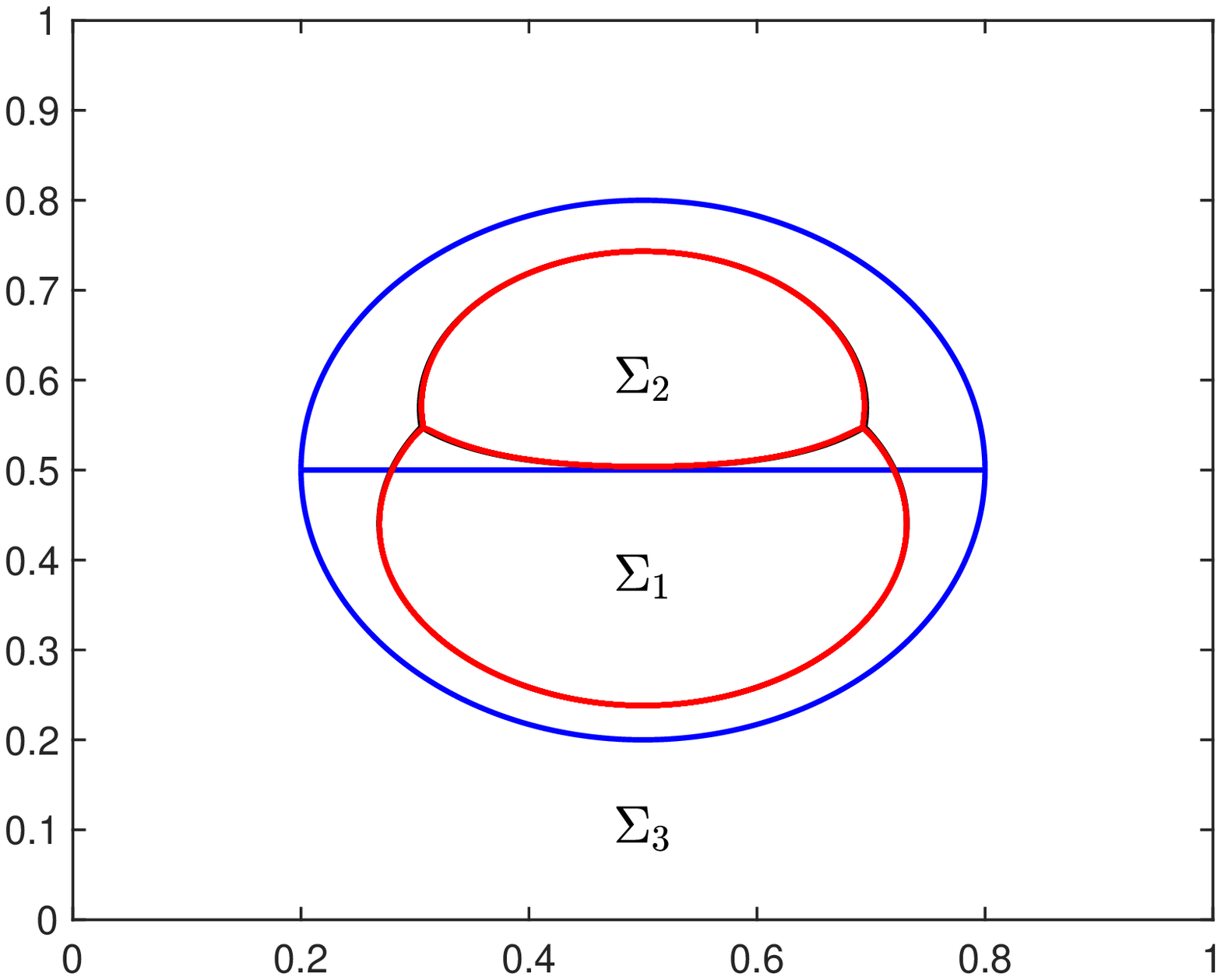}\label{fig:front_tracking_1}}
\hspace{2ex} 
\subfigure[]{
\includegraphics[width=0.4\textwidth]{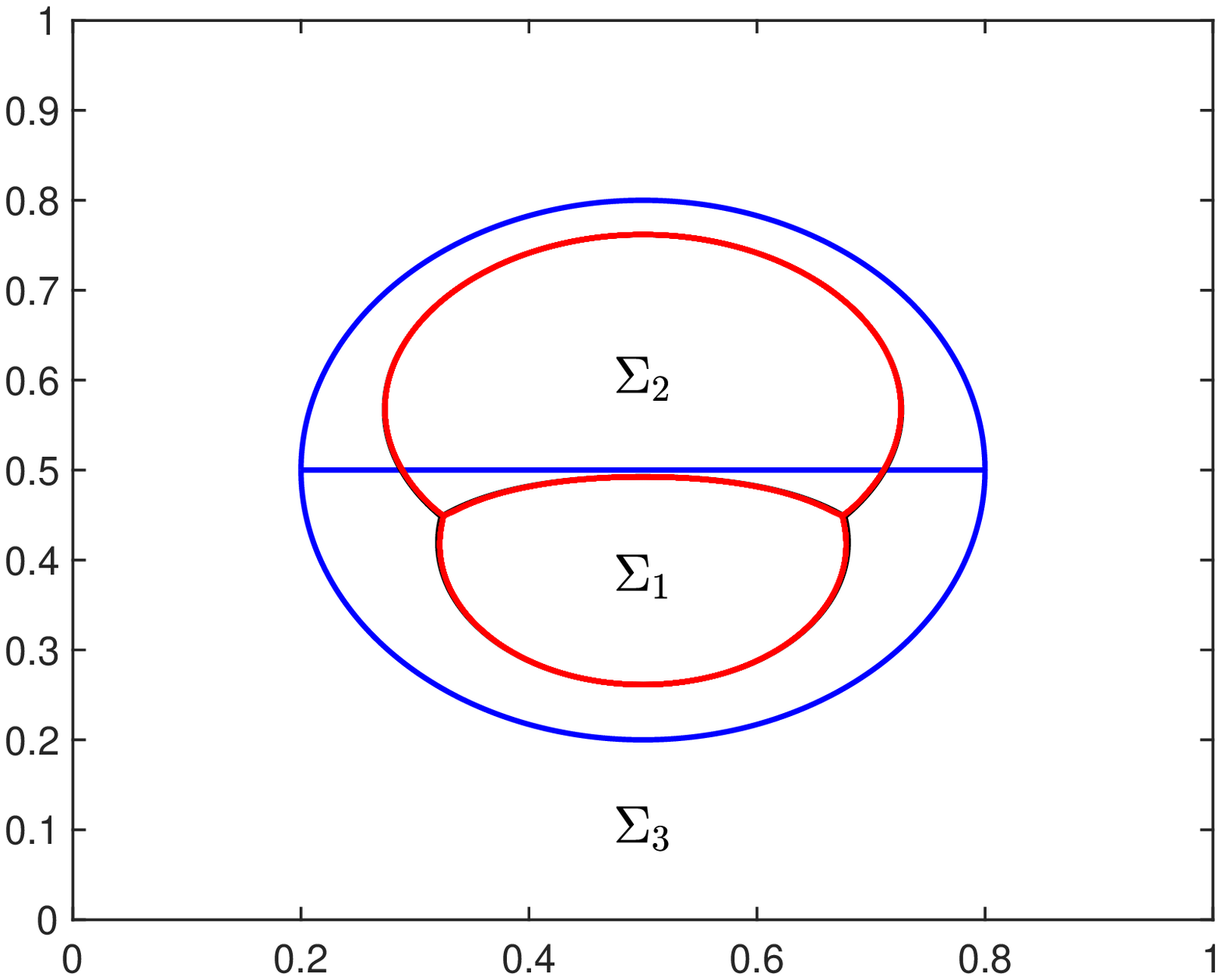}\label{fig:front_tracking_2}}
\caption{Evolution of a three-phase configuration where the blue curve shows the initial condition for two different choices of surfaces tensions. Final configuration computed using threshold dynamics Algorithm \ref{alg:ES} (red), compared to the benchmark result computed using front tracking (black).}
\label{fig:front_tracking}
\end{figure}

\subsection{Failure of the algorithm}

In this section, we present an example, in the spirit of the discussion of Section 5, for which the algorithm fails: wetting occurs when kernels obtained from the proposed construction are used, even though the desired sharp interface model is well-posed (the given surface tensions satisfy the triangle inequality).
Consequently, angles formed at the triple junction are not the naively expected ones, and the evolution differs from the intended dynamics.

\begin{example}\label{ex:wetting}
We revisit Example \ref{ex:grim_reaper_2}. Notice that $\mu_{1,2}$ does not affect the evolution since the interface $\Gamma_{1,2}$ has curvature zero. Consider then Example 6.2 but with $\mu_{1,2} = \mu_{2,1} = 1/2$ ($\frac{1}{\mu}$ is still conditionally negative definite). In this case wetting may occur as $a_{i,j} +b_{i,j}$ do not satisfy the triangle inequality and, in addition, $\Gamma$-convergence is not established.  Notice how this example exhibits the same features of Example \ref{ex:counterexample} described in \autoref{sec:nonconvergence}. Indeed, numerical simulations show that wetting occurs: the algorithm instantaneously nucleates phase 1 along the interface $\Gamma_{2,3}$. This thin layer of phase 1 remains between phases 1 and 2 throughout the evolution. Its thickness appears to depend on the time step size and scale as $\sqrt{\delta t}$. See Figures \ref{fig:wetting_phases_initial_vs_final} and \ref{fig:wetting_phases} and Tables \ref{table:ex_wetting} and \ref{table:ex_wetting_layer}. It is worth mentioning that when $\alpha$ and  $\beta$ are chosen to also satisfy \eqref{eq:choice_alpha_beta_triangle_inequality}, which guarantees no wetting, convergence is observed.
\end{example}


\begin{table}[htp]
\centering
\footnotesize
\begin{tabular}{cccc}
\hline
 & \multicolumn{3}{l}{Errors and order, Example \ref{ex:wetting}} \\ \cline{2-4}
$\#$ Time steps	& $\#$ Grid points	& Error	& Conv. rate\\
\hline
100 & $128 \times 128$ & \num{6.4819e-02} & - \\
200 & $256 \times 256$ & \num{1.8341e-02} & 1.82 \\
400 & $512 \times 512$ & \num{1.4839e-02} & 0.31 \\
800 & $1024 \times 1024$ & \num{1.5284e-02} & -0.04 \\
1600 & $2048 \times 2048$ & \num{1.7998e-02} & -0.24 \\
3200 & $4096 \times 4096$ & \num{1.9145e-02} & -0.09 \\
6400 & $8192 \times 8192$ & \num{1.9731e-02} & -0.04 \\
\end{tabular}
\caption{Errors are computed as the $l_\infty$ norm of the $l_1$ error of the characteristic functions of each phase for Example \ref{ex:wetting}.}
\label{table:ex_wetting}
\end{table}

\begin{table}[htp]
\centering
\footnotesize
\begin{tabular}{cccc}
\hline
 & \multicolumn{3}{l}{Errors and order, Example \ref{ex:wetting}} \\ \cline{2-4}
$\#$ Time steps	& $\#$ Grid points	& Area	& Conv. rate\\
\hline
100 & $128 \times 128$ & \num{1.6357e-02} & - \\
200 & $256 \times 256$ & \num{6.4087e-03} & 1.35 \\
400 & $512 \times 512$ & \num{3.7994e-03} & 0.75 \\
800 & $1024 \times 1024$ & \num{2.4853e-03} & 0.61 \\
1600 & $2048 \times 2048$ & \num{1.7276e-03} & 0.52 \\
3200 & $4096 \times 4096$ & \num{1.2108e-03} & 0.51 \\
6400 & $8192 \times 8192$ & \num{8.5044e-04} & 0.51 \\
\end{tabular}
\caption{Area of the thin layer formed along phases 2 and 3 in Example \ref{ex:wetting}.}
\label{table:ex_wetting_layer}
\end{table}


\begin{figure}[htp]
\centering
\subfigure{\includegraphics[width=0.45\textwidth]{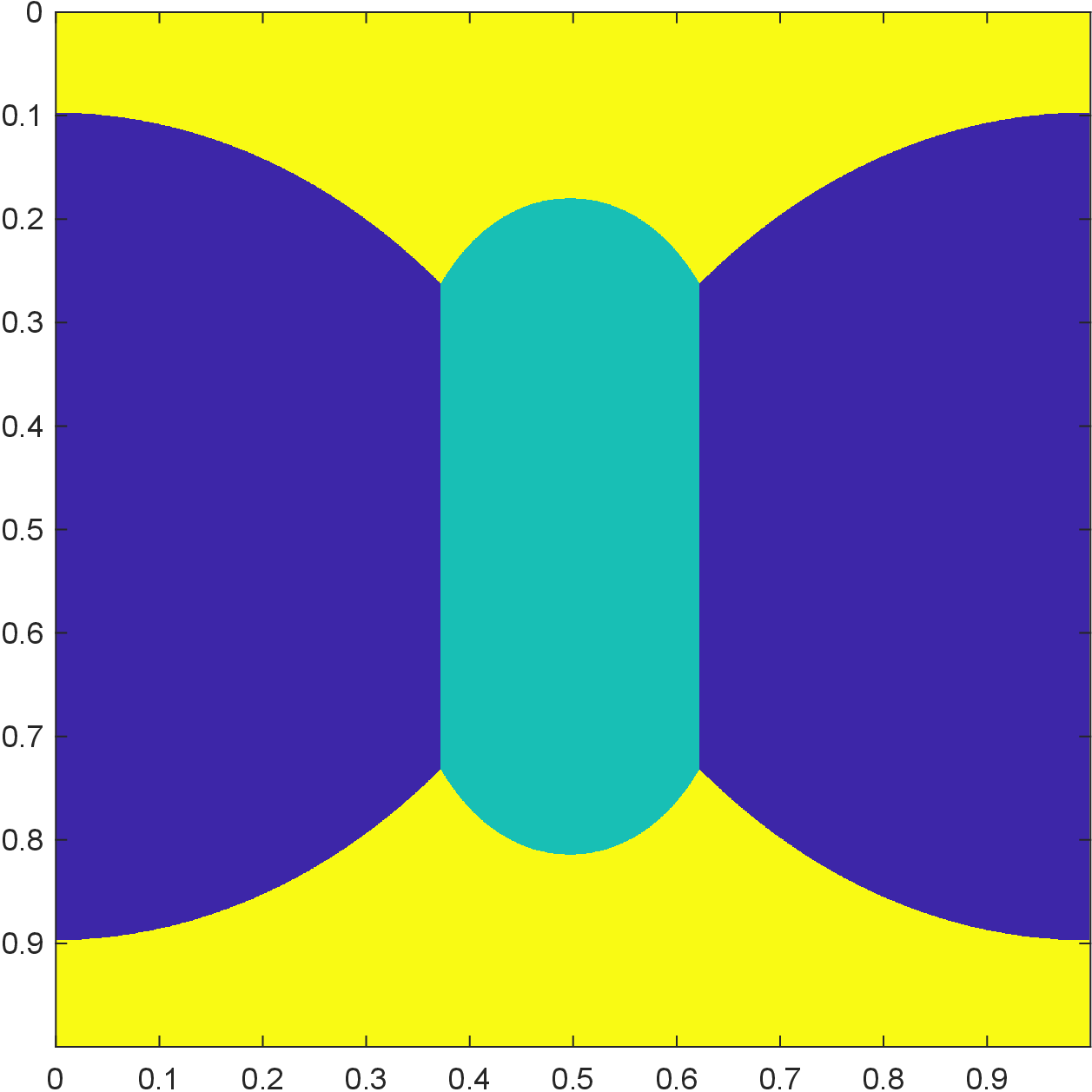}\label{fig:ex4_wetting_phases_initial}}
\subfigure{\includegraphics[width=0.45\textwidth]{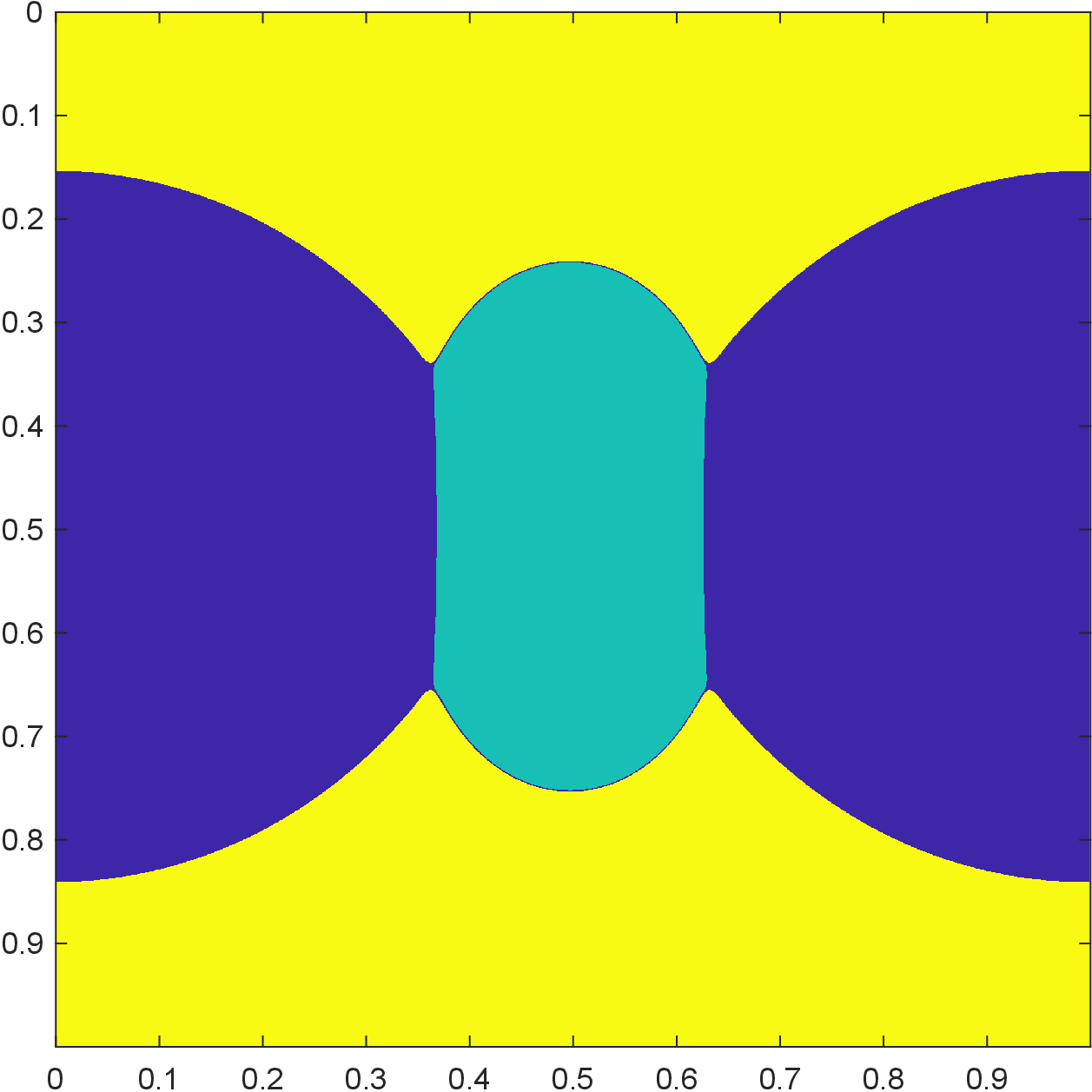}\label{fig:ex4_wetting_phases_final}}
\caption{Example of wetting in the multiphase setting even when the surface tensions satisfy the triangle inequality and the kernels are positive: \emph{Left:} In the initial condition there is no phase 1 between phases 2 and 3. \emph{Right:} Algorithm \ref{alg:ES} immediately nucleates a thin layer of phase 1 along the $\Gamma_{2.3}$ interface present in the initial condition. That thin wetting layer of phase 1, shown as the darkest region, remains between phases 2 and 3 throughout the evolution.}
\label{fig:wetting_phases_initial_vs_final}
\end{figure}

\begin{figure}[htp]
\centering
\subfigure[$n=128$]{
\includegraphics[width=0.3\textwidth]{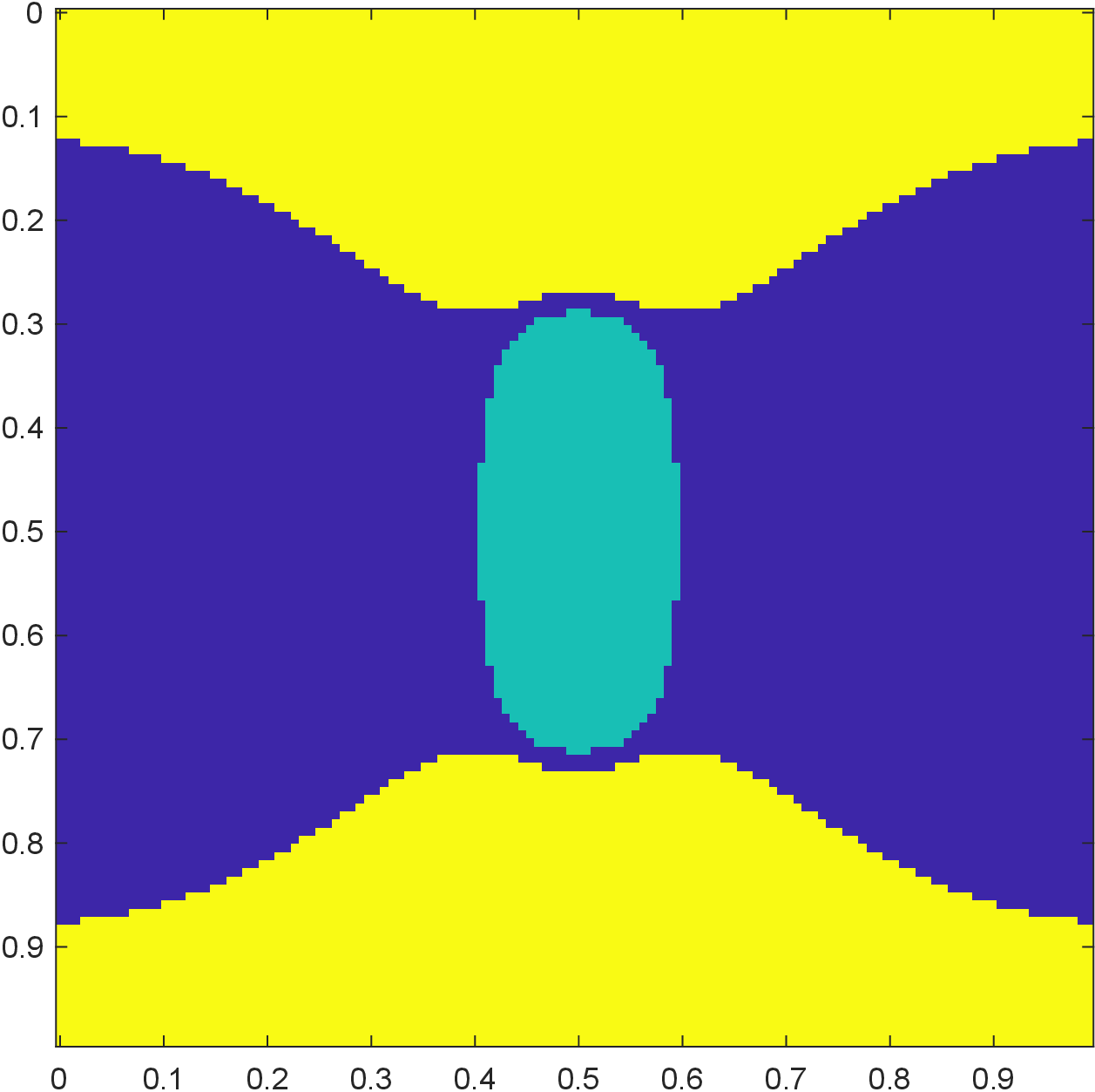}\label{fig:wetting_phases1}}
\subfigure[$n=256$]{
\includegraphics[width=0.3\textwidth]{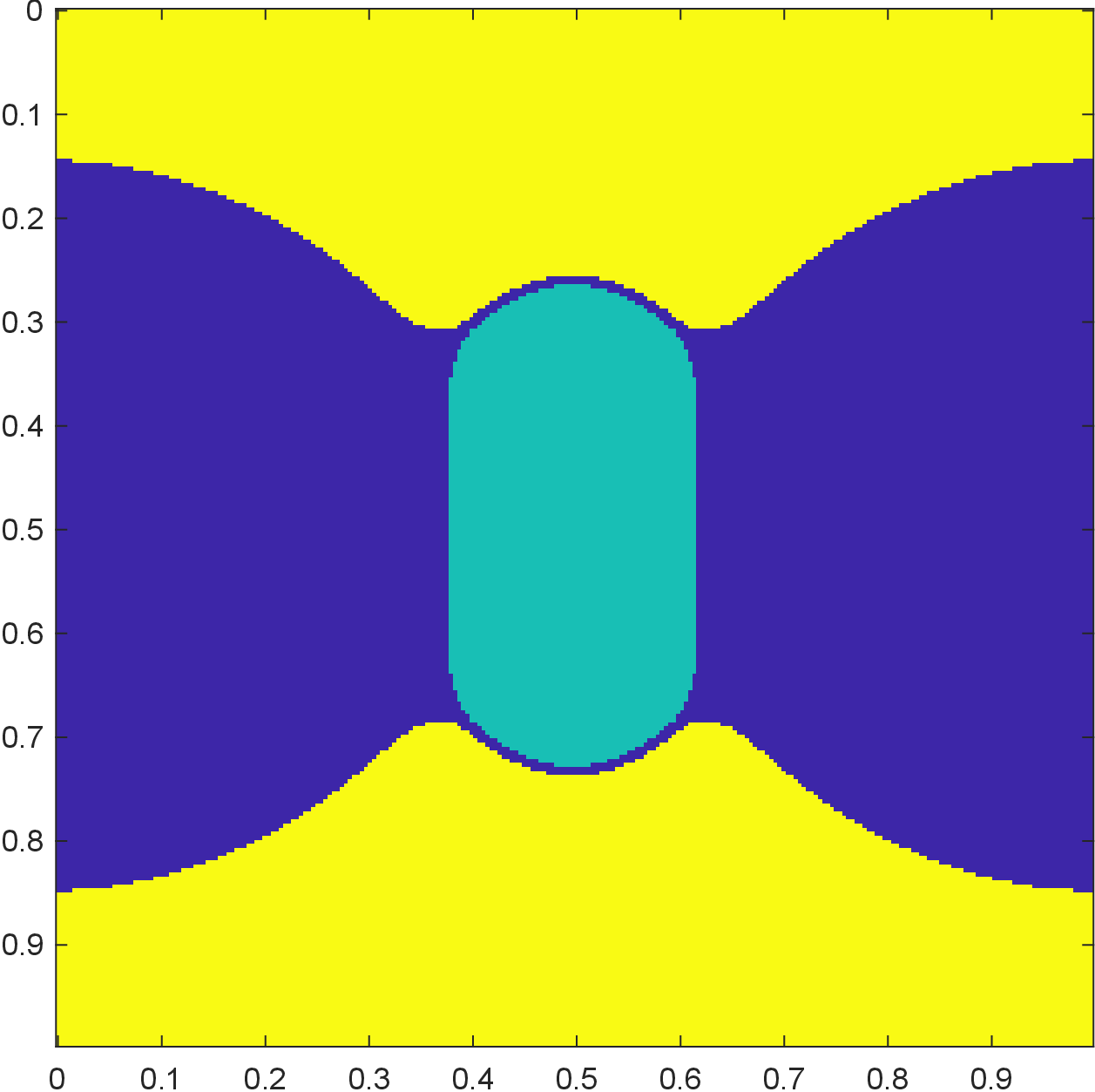}\label{fig:wetting_phases2}}
\subfigure[$n=512$]{
\includegraphics[width=0.3\textwidth]{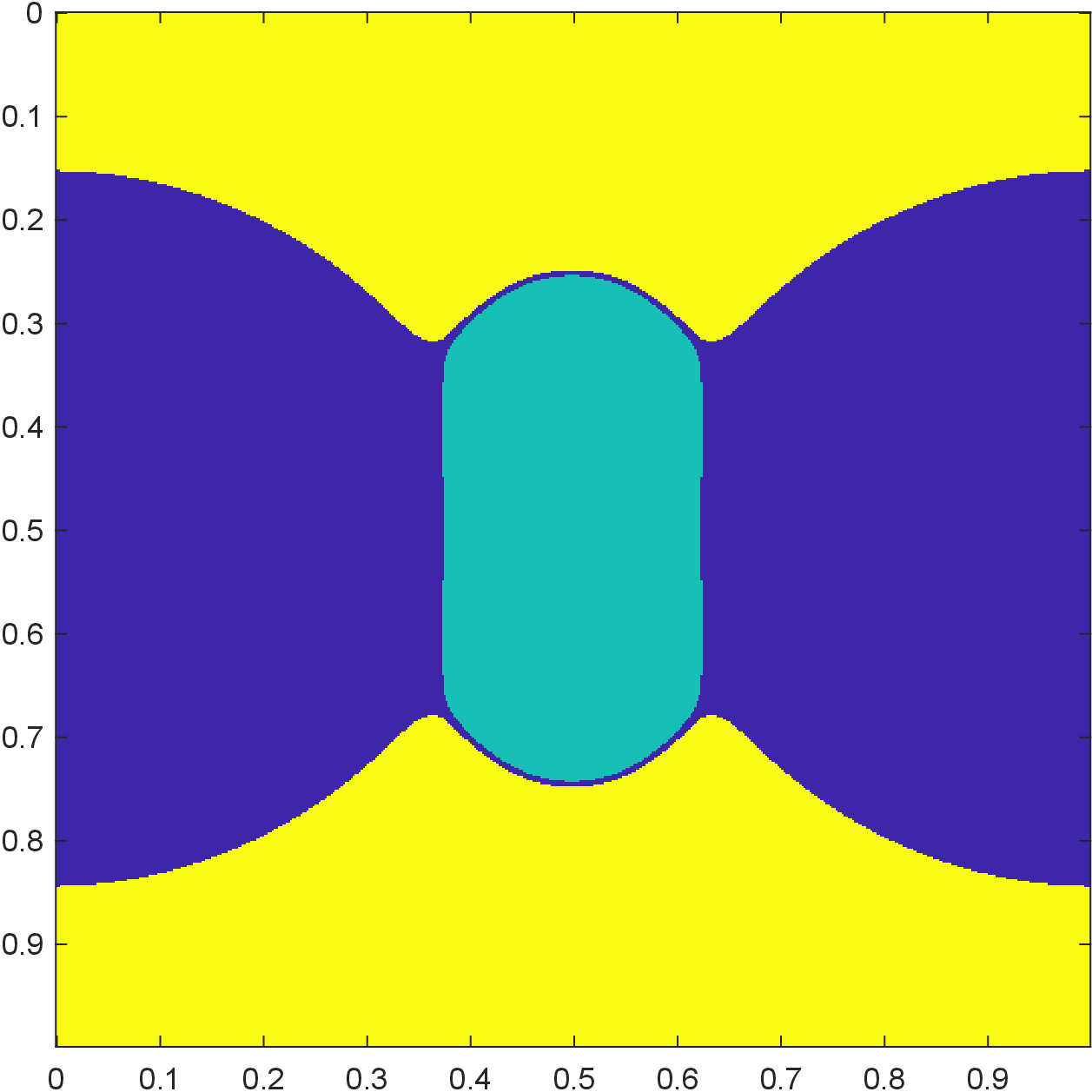}\label{fig:wetting_phases3}}
\subfigure[$n=1024$]{
\includegraphics[width=0.3\textwidth]{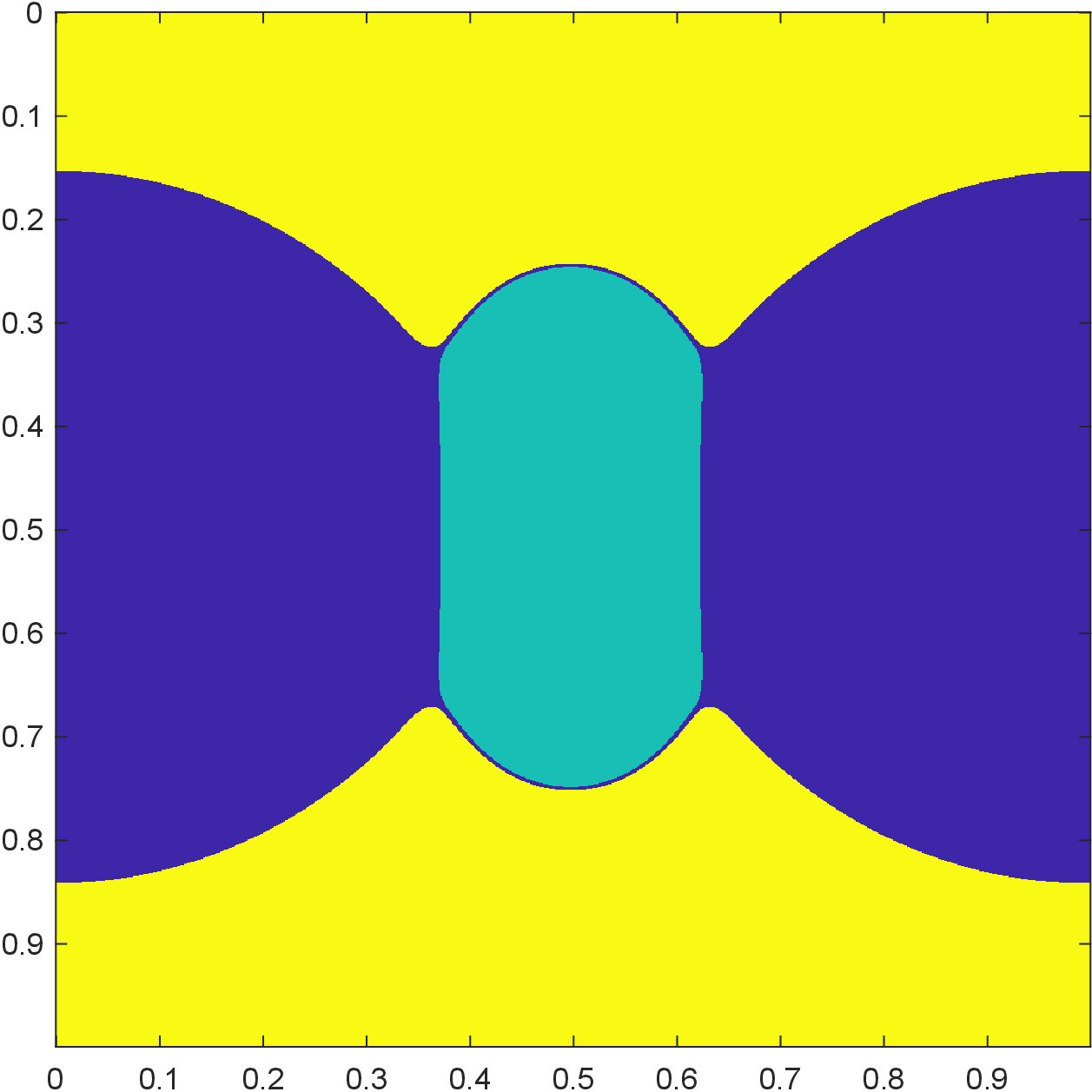}\label{fig:wetting_phases4}}
\subfigure[$n=2048$]{
\includegraphics[width=0.3\textwidth]{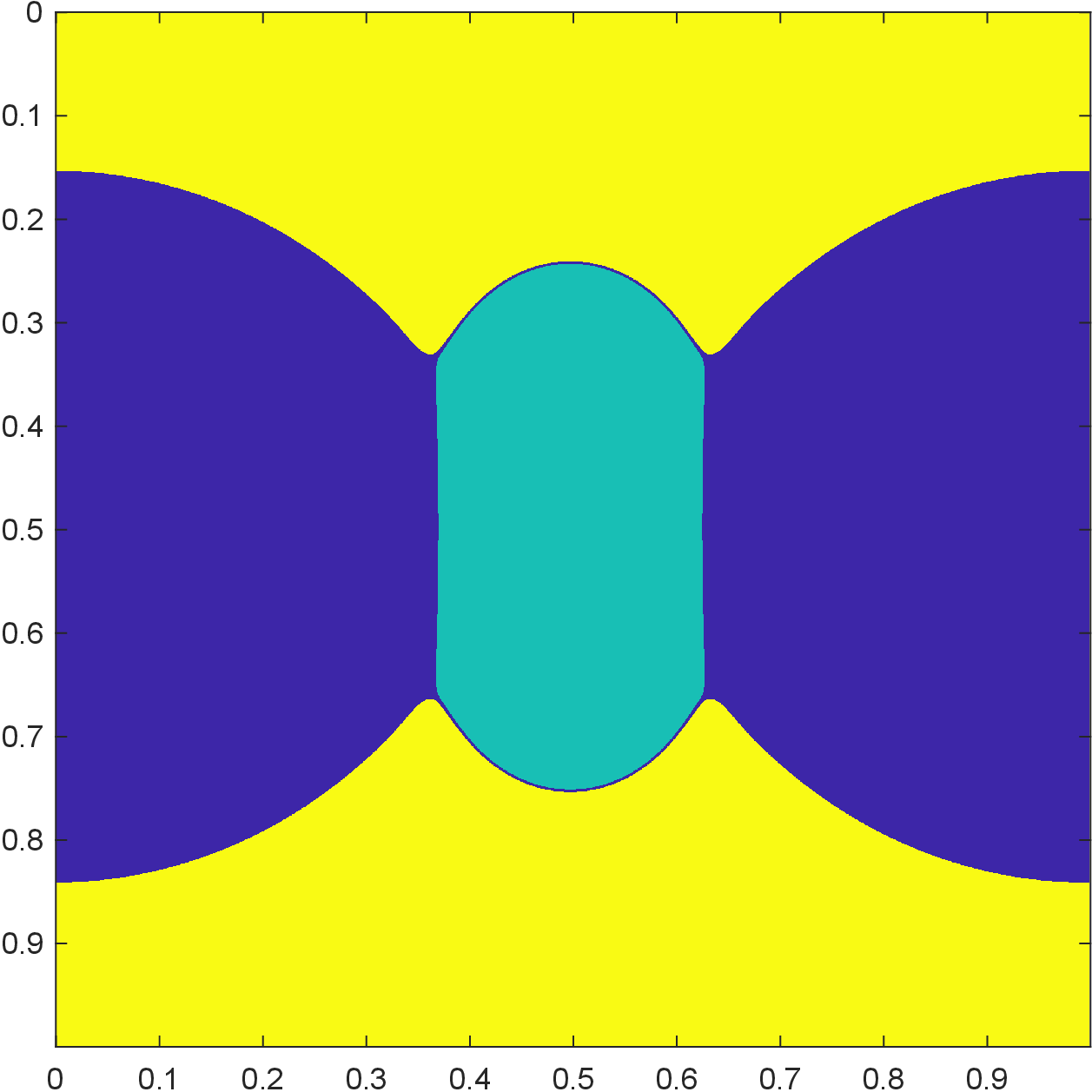}\label{fig:wetting_phases5}}
\subfigure[$n=4096$]{
\includegraphics[width=0.3\textwidth]{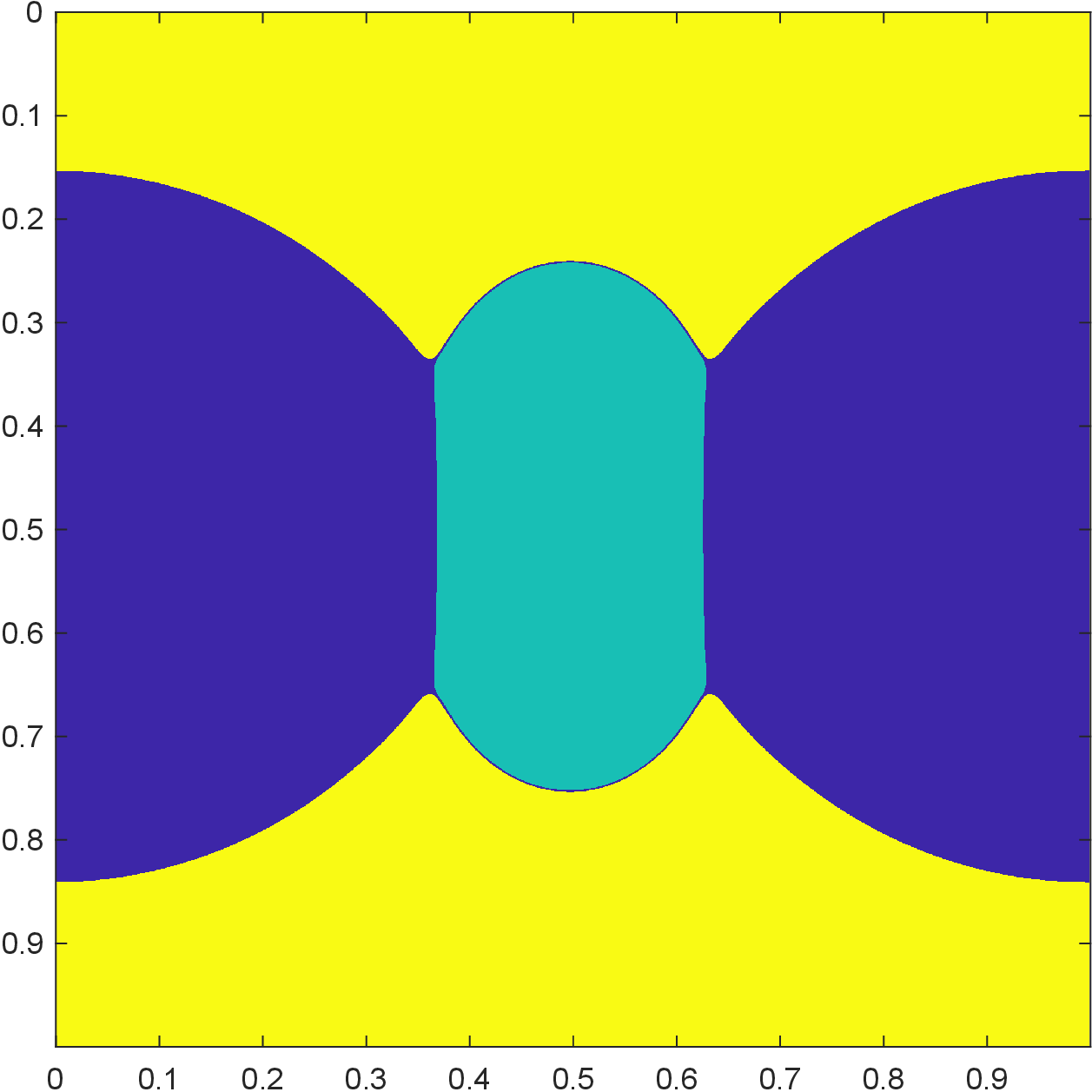}\label{fig:wetting_phases6}}
\caption{Example of wetting in the multiphase setting even when the surface tensions satisfy the triangle inequality and the kernels are positive.}
\label{fig:wetting_phases}
\end{figure}

\section{Conclusions}

We presented a simple and efficient algorithm for the mean curvature flow of a general $N$-phase network where the $N\choose 2$ isotropic surface tensions and $N\choose 2$ isotropic mobilities can be individually specified.
We showed that the algorithm is unconditionally gradient stable under mild conditions on the surface tensions and mobilities, which are satisfied, for instance, in the Read-Shockley model.
The $\Gamma$-convergence of the underlying approximate energies to the desired limit gives confidence that the new algorithm converges to the correct dynamics for important classes of surface tensions and mobilities that are commonly employed in simulations by materials scientists. 
However, we also presented counter examples to the convergence of the algorithm when kernels obtained from the proposed construction are used under certain conditions, indicating limitations to our current understanding of this class of numerical methods.

\section{Acknowledgements}
\noindent The authors gratefully acknowledge support from the NSF grant DMS-1719727.

\bibliographystyle{amsalpha}
\bibliography{bibliography}

\end{document}